\documentclass{amsart}
\usepackage{amsthm,amssymb}
\usepackage{verbatim}
\usepackage{bbm}
\usepackage{hyperref}
\usepackage[all]{xy}
\usepackage{amsmath,amsfonts}	
\usepackage{amsthm,latexsym,amssymb}
\usepackage{color}
\usepackage{tikz}
\usepackage[colorinlistoftodos]{todonotes}

\newtheorem{thm}{Theorem}
\newtheorem{prop}[thm]{Proposition}
\newtheorem{lem}[thm]{Lemma}

\newtheorem{cor}[thm]{Corollary}

\theoremstyle{definition}
\newtheorem{defn}[thm]{Definition}
\newtheorem{ex}[thm]{Example}
\newtheorem{rem}[thm]{Remark}

\makeatletter
\def\x@arrow{\DOTSB\Relbar}
\def\xlongequalsignfill@{\arrowfill@\x@arrow\Relbar\x@arrow}
\newcommand{\xlongequal}[2][]{%
        \ext@arrow 0099\xlongequalsignfill@{#1}{#2}}
\def\xlongrightarrowfill@{\arrowfill@\relbar\relbar\longrightarrow}
\newcommand{\xlongrightarrow}[2][]{%
        \ext@arrow 0099\xlongrightarrowfill@{#1}{#2}}

\renewcommand{\bar}{\overline}

\renewcommand{\[}{\begin{equation*}}
\renewcommand{\]}{\end{equation*}}

\newcommand{\e}{\sqrt{-1}}

\def\scrF{\mathcal{F}}

\def\spinc{$Spin^c$}

\begin{document}
\title[ ]{Seiberg--Witten invariants on manifolds with Riemannian foliations of codimension 4}
\author{Yuri Kordyukov}\thanks{The first named author was supported by the grant no. 16-01-00312 from the Russian Foundation
of Basic Research.}
\address{Institute of Mathematics, Russian Academy of Sciences, 112 Chernyshevsky str. 450008 Ufa, Russia} \email{yurikor@matem.anrb.ru} 
\author{Mehdi Lejmi}\thanks{The second named author acknowledges travel support from the Communaut\'{e} fran\c{c}aise de Belgique via an ARC and from the Belgian federal government via the PAI ``DyGeSt"} 
\address{Department of Mathematics, Bronx Community College of CUNY, Bronx, NY 10453, USA.}
\email{mehdi.lejmi@bcc.cuny.edu}
\author{Patrick Weber}\thanks{The third named author was supported by an Aspirant grant from the F.R.S.-FNRS and acknowledges travel support from the Communaut\'{e} fran\c{c}aise de Belgique via an ARC and from the Belgian federal government via the PAI ``DyGeSt".}
\address{D\'epartement de Math\'ematiques, Universit\'e Libre de Bruxelles CP218, Boulevard du Triomphe, Bruxelles 1050, Belgique. }
\email{pweber@ulb.ac.be}

\begin{abstract}
We define Seiberg--Witten equations on closed manifolds endowed with a Riemannian foliation of codimension $4$. When the foliation is taut, we show compactness of the moduli space under some hypothesis
satisfied for instance by closed K-contact manifolds.
Furthermore, we prove some vanishing and non-vanishing results and we highlight that the invariants may be used to
distinguish different foliations on diffeomorphic manifolds.
\end{abstract}

\maketitle

\section{introduction}

Seiberg--Witten invariants are one of the main tools in the study of the differential topology of 4-manifolds. Since the foundational paper~\cite{Wit}, a lot of work has been done to apply this theory to various aspects of three and four-dimensional geometry. 
A natural idea is to extend this framework to higher-dimensional manifolds. However the Seiberg--Witten equations then turn out to be overdetermined and one needs to assume that the manifold carries an additional geometric structure to extract useful information. For some results in this direction see~\cite{Bil-Der-Koc,Deg-Bul,Deg-Ozd,Gao-Tia}.


This article lays the groundwork for the case in which the higher-dimensional manifold admits a Riemannian foliation of codimension~$4$. For Yang--Mills instantons this has already been investigated by Wang~\cite{Wan} and, to a lesser extent, by Baraglia and Hekmati~\cite{Bar-Hek} in the more restrictive context of K-contact 5-manifolds. 

The main technical ingredient needed when extending to Riemannian foliations is the theory of basic transversally elliptic operators~\cite{Kac}. Most proofs become adaptations of standard arguments (see for instance~\cite{Mor, nic, Sal} or the excellent survey of Li~\cite{Li}) but we decided to include them nevertheless in order to stay self-contained and to emphasize the modifications which have to be done due to the presence of the mean curvature form of the foliation.
The main result of the paper is the compactness of the moduli space when the foliation is taut, and under a technical hypothesis satisfied for instance by closed K-contact manifolds. Furthermore, we give a non-vanishing result for transversally K\"ahler foliations and
we highlight that the invariants may be used to distinguish different foliations on diffeomorphic manifolds. 
As the present work advanced, we became aware that J.~Lee and A.~Renner have work in progress on related results according to~\cite{Wan}.

The plan of the article is the following: Sections~\ref{Sec_2},~\ref{Sec_3} and~\ref{Sec_4} are of an expository nature, setting up the notions of transverse structures and basic forms needed later on. We also establish a Weitzenb\"{o}ck formula as done in~\cite{Gla-Kam}. After these preliminaries, in Section~\ref{Sec_5} we write down a basic version of the Seiberg--Witten equations and in Section~\ref{Sec_6}, we state a basic version of Sobolev estimates. In Section~\ref{Sec_7}, we attack the crucial compactness property of the moduli space whilst Section~\ref{Sec_8} deals with other properties of the moduli space such as transversality and orientability and eventually gives the definition of basic Seiberg--Witten invariants. Applying the Weitzenb\"{o}ck formula, we deduce in Section~\ref{Sec_9} that, under some curvature assumptions, all solutions of the basic Seiberg--Witten equations have trivial spinor. Interestingly enough, on non-taut foliations, the mean curvature form persists and hence gives a non-trivial difference with the classical four-dimensional case. On the other hand, as expected, the equations simplify considerably in the presence of a transversal K\"{a}hler structure. This allows us to prove in Section~\ref{Sec_10} a non-vanishing theorem for basic Seiberg--Witten invariants. The most well-known examples of transversally K\"{a}hler foliations being Sasakian manifolds, we take a quick look in Section~\ref{Sec_11} at this class and their almost-K\"{a}hler counterpart: K-contact manifolds. We conclude this article by giving possible applications of the basic Seiberg--Witten invariants. We hope that this will inspire future work.

%
\medskip

{\bf Acknowledgements: } We warmly thank Ken Richardson for his invaluable help. 
We are very grateful to Shuguang Wang for pointing out some mistakes in an earlier draft as well as letting us know about his work in progress with his student Andrew Renner on elliptic estimates for basic sections. 
We thank Georges Habib, Charles Boyer and Weimin Chen for very useful discussions and Tian-Jun Li for his interest in this work and comments.
Finally, we thank Joel Fine for proposing the problem and his support. 

%

\section{Transverse structures}\label{Sec_2}

Let $M$ be a closed oriented smooth manifold of dimension $n$ endowed with a {\it{Riemannian foliation}} $\scrF$
of codimension $m$ and $g$ a {\it{bundle-like metric}} on~$M$~\cite{Rei}.
The metric $g$ gives an orthogonal splitting of the tangent bundle $TM=T\scrF\oplus T\scrF^\perp$. Moreover we have
a direct sum
\[
g=g_{T\scrF}\oplus g_{T\scrF^\perp}.
\] 
The normal bundle $Q=TM/T\scrF$ is identified with
$T\scrF^\perp$ and so the metric $g_{T\scrF^\perp}$ induces a Riemannian metric $g_{Q}$ on $Q$.
Since the metric $g$ is bundle-like, the metric~$g_{Q}$ satisfies
\[
\mathcal{L}_{X} g_{Q}=0,\quad{\text{for all }}X \in \Gamma(T\scrF),
\]
where $\mathcal{L}$ is the Lie derivative.

A bundle-like metric assures the existence of a unique torsion-free metric connection $\nabla^T$
called the \textit{transverse Levi-Civita connection}~\cite{Mol,Rei,Ton} on the normal bundle~$Q$. 
Explicitly, let $D^g$ be the Levi-Civita connection associated to the metric~$g$ and denote by $\pi:TM\longrightarrow Q$
the projection operator. We define $\nabla^T$ on $Q$ for $s\in\Gamma(Q)$ by
\[
\nabla^T_X s:=\left\{\begin{array}{ccl}\pi\left([X,Y_s]\right),&&\quad X\in\Gamma(T\scrF), \\
 \pi\left(D^g_X Y_s\right),&&\quad X\in\Gamma(Q),\end{array}\right.
\]
where $Y_s\in\Gamma(T\scrF^\perp)$ corresponds to $s$ via the identification $Q\cong T\scrF^\perp.$
The torsion-free condition and the Jacobi identity imply that $\nabla^T$ is a {\it{leafwise flat connection}} i.e.\! its curvature
$R^{T}$ satisfies 
\[
R^{T}(X,Y) :=\nabla^T_X\nabla^T_Y-\nabla^T_Y\nabla^T_X-\nabla^T_{[X,Y]}= 0, \quad \text{for all }X,Y \in \Gamma(T\scrF).
\]
Actually, it can be shown that $R^T$ satisfies the stronger condition 
\[
\iota_XR^{T}= 0, \quad \text{for all }X\in \Gamma(T\scrF),
\]
where $\iota$ is the contraction operator (see for example \cite{Habib-thesis}).
We define {\textit{the transversal Ricci curvature}} $Ric^T:\Gamma(Q)\rightarrow \Gamma(Q)$ and the
{\textit{transversal scalar curvature}} $s^T$ by
\[
Ric^T(X):=\sum_{i=1}^mR^T(X,e_i)e_i,\quad s^T:=\sum_{i=1}^mg_{Q}(Ric^T(e_i),e_i),
\]
where $\{e_1,\cdots,e_m\}$ is a local $g_{Q}$-orthonormal frame of $Q.$

The {\it{mean curvature vector field}} $\tau$ of $\scrF$ is defined by
\[
\tau:=\sum_{i=1}^{\dim\scrF}\pi(D^g_{\xi_i}\xi_i),
\]
where $\{\xi_1,\cdots,\xi_{\dim\scrF}\}$ is a local $g_{T\scrF}$-orthonormal basis of $T\scrF.$
The \textit{mean curvature form} $\kappa$ is the $g_Q$-Riemannian dual of $\tau.$
In other words, the mean curvature form is the trace of the second fundamental form of the foliation, which itself is symmetric by integrability of the distribution.
A Riemannian foliation $\scrF$ is called \textit{taut} if there exists a bundle-like metric for which $\kappa=0$.

The connection $\nabla^T$ induces a connection, still denoted by $\nabla^T$, on the bundle~$\Lambda^rQ^\ast.$
The {\it{basic sections}} $\Gamma_B(\Lambda^r Q^\ast)$ of $\Lambda^r Q^\ast$ are then defined by
\[
 \Gamma_B(\Lambda^r Q^\ast):=\{\alpha\in\Gamma(\Lambda^rQ^\ast)\,|\,\nabla^T_X\alpha=0,\,\,\forall X\in\Gamma(T\scrF)\}.
 \]
We can identify $ \Gamma_B(\Lambda^rQ^\ast)\subset\Omega^r(M):=\Gamma(\Lambda^r T^\ast(M))$
with the set of {\it{basic forms}} 
\[
\Omega_B^r(M):=\{\alpha\in\Omega^r(M)\,|\,\iota_X(\alpha)=0,\iota_X(d\alpha)=0,\,\,\forall X\in\Gamma(T\scrF)\},
\]
where $d$ is the exterior derivative. We have the orthogonal decomposition~\cite{Alv}
$$ \Omega^r(M)= \Omega^r_B(M) \oplus \Omega^r_B(M)^\perp,$$
with respect to the $C^\infty$-Fr\'echet topology.
In particular, the mean curvature form can be written as
$$\kappa = \kappa_B + \kappa_0,$$
where $\kappa_B$ is the basic part and $\kappa_0$ its orthogonal complement.
Moreover, it is shown in~\cite{Alv} that the basic part of the mean curvature is closed, i.e.\! $d\kappa_B=0$.

The metric $g_Q$ induces a {\it{transverse Hodge-star operator}}~\cite{Ton}
\[
\bar\ast:\Lambda^rQ^\ast\longrightarrow\Lambda^{m-r}Q^\ast.
\]
Since the metric $g$ is bundle-like, the transverse Hodge-star operator preserves basic forms
\[
\bar\ast:\Omega_B^r(M)\longrightarrow\Omega_B^{m-r}(M).
\]
It is related to the usual Hodge-star operator $\ast$ (induced by the metric $g$) via
\[
\bar\ast\,\alpha=(-1)^{(m-r)\dim\scrF}\ast\left(\alpha\wedge\chi_\scrF\right),
\]
where $\alpha\in\Omega_B^r(M)$ and $\chi_\scrF$ is the {\it{characteristic form}} of the foliation $\scrF$~\cite{Ton}.
We remark that we have the (basic) inner product in $\Omega^r_B(M)$
\begin{equation}\label{inner_product}
\langle\alpha,\beta\rangle=\int_M\alpha\wedge\bar\ast\beta\wedge\chi_\scrF,
\end{equation}
which is the restriction of the usual inner product on $\Omega^r(M)$ to $\Omega_B^r(M)$~\cite{Ton}.

Note that the exterior derivative $d$ preserves basic forms. Define $$d_B := d |_{\Omega_B^r(M)}: \Omega_B^r(M) \longrightarrow \Omega_B^{r+1}(M)$$ to be the restriction of $d$ to basic forms. 
The resulting complex of basic forms is a subcomplex of the de Rham complex.
Its cohomology is called the \textit{basic cohomology} $H^r(\scrF)$ and can be interpreted as some kind of de Rham cohomology on the leaf space.
It turns out that for Riemannian foliations on closed manifolds the basic cohomology groups are all finite-dimensional~\cite{Kac-Hec-Ser,Par-Ric}.

In terms of the transverse Hodge-star operator, the adjoint operator $\delta_B$ of $d_B$ with respect to the (basic) inner product~(\ref{inner_product}) is given by
\begin{equation}\label{codiff_exp}
\delta_B\,\alpha=(-1)^{m(r+1)+1}\,\bar\ast\,(d_B- \kappa_B\wedge)\,\bar\ast\,\alpha,
\end{equation}
where $\alpha\in\Omega_B^r(M).$
In a local orthonormal basic frame $\{e_1,\cdots,e_m\}$ of $\Gamma(Q)$ we can write
\begin{equation}\label{diff_codif_expression}
d_B=\sum_{i=1}^me_i^{\flat_{g_Q}}\wedge\nabla^T_{e_i} {\text{ and its codifferential }}\delta_B=-\sum_{i=1}^m\iota_{e_i}\nabla^T_{e_i}+\iota_{\tau_B},
\end{equation}
where $\tau_B^{\flat_{g_Q}}=\kappa_B$ and $\flat_{g_Q}$ denotes the $g_Q$-Riemannian dual~\cite{Alv,Jun}. 
One can then define the {\textit{basic Laplacian}}
\[
\Delta_B:=d_B\delta_B+\delta_Bd_B.
\]
Since the basic Laplacian is a {\textit{basic transversally elliptic}} differential operator (see Section~\ref{Sec_6}),
it is a Fredholm operator with a finite index~\cite{Kac}.
Moreover, using the Hodge decomposition of the basic Laplacian~\cite{Kam-Ton-1},
we have that $\dim H^r(\scrF)$ is equal to the (finite) dimension of $\Delta_B$-harmonic basic $r$-forms. In general, $\Delta_B$ does {\it{not}} commute
with the transverse Hodge operator $\bar\ast$ whereas when $\kappa_B=0$ we do have $\Delta_B\,\bar\ast=\bar\ast\,\Delta_B.$ 

\begin{rem}\label{harmonic_remark}
It is known that the space of bundle-like metrics on $(M,\scrF)$ is infinite dimensional~\cite{San}. 
It can be shown that any bundle-like metric can be deformed in the leaf directions leaving the transverse part untouched in such a way that the mean curvature form becomes basic~\cite{Dom}. 
Moreover, another suitable conformal change of the metric in the leaf direction will make $\kappa_B$, the basic component of $\kappa$, $\Delta_B$-harmonic~\cite{Mar-Min-Ruh}.
\end{rem}

\begin{rem}\label{periodic_form}
We remark that, for any foliation $\scrF$ on a connected manifold, $H^1(\scrF)$ is a subgroup of $H^1(M,\mathbb{R})$~\cite[Proposition 2.4.1]{bo-ga}.
Moreover, if $\kappa_B=0,$ then it readily follows from~(\ref{diff_codif_expression}) that any basic $\Delta_B$-harmonic form is $\Delta^g$-harmonic, where~$\Delta^g$ is the (usual) Riemannian Laplacian associated to the metric $g.$ 
\end{rem}

When the foliation $\scrF$ has codimension $m=4$, we get a decomposition
\[
\Lambda^2Q^\ast=\Lambda^+Q^*\oplus\Lambda^-Q^*,
\]
where $\Lambda^\pm$ corresponds to the $(\pm 1)$-eigenvalue under the action of the transverse Hodge-star operator $\overline\ast.$
Since the metric is bundle-like, there is an induced splitting of basic $2$-forms
\[
\Omega^2_B(M)=\Omega^+_B(M)\oplus\Omega^-_B(M).
\]
Let $\mathcal{H}^+(\scrF)$ (resp. $\mathcal{H}^-(\scrF)$) be the space of $\Delta_B$-harmonic $\bar\ast$-self-dual
(resp. $\bar\ast$-anti-self-dual) basic $2$-forms. When $\kappa_B=0,$ we have $H^2(\mathcal{F})\cong \mathcal{H}^+(\scrF)\oplus \mathcal{H}^-(\scrF).$

We can consider the following {{basic transversally elliptic}} linear operator 
\begin{equation*}
\delta_B\oplus d^+_B:\Omega^1_B(M)\longrightarrow\Omega^0_B(M)\oplus\Omega^+_B(M),
\end{equation*}
where $d^+_B:=\frac{1}{2}(1+\bar\ast)\,d_B$ is the composition of $d_B$ with the projection onto $\Omega^+_B(M).$ 
In particular, the kernel and the cokernel of $\delta_B\oplus d^+_B$ are finite-dimensional.
\begin{prop}\label{complex}
The cohomology groups of the transversally elliptic complex~\cite{Wan}
\begin{equation*}
0\rightarrow\Omega^0_B(M)\xlongrightarrow{d_B}\Omega^1_B(M)\xlongrightarrow{d_B^+}\Omega^+_B(M)\rightarrow 0
\end{equation*}
are respectively $\mathbb{R},\{\alpha\in\Omega^1_B(M)\,|\,d_B^+\alpha=\delta_B\alpha=0\}$ and $\{\beta\in\Omega^+_B(M)\,|\,(d_B-\kappa_B\wedge)\beta=\delta_B\beta=0\}.$
In the particular case $\kappa_B=0$, the groups 
are respectively~$\mathbb{R}$, $H^1(\scrF)$ and $\mathcal{H}^+(\scrF).$
\end{prop}
\begin{proof}
First, let $\alpha\in\Omega^1_B(M)$ be in the kernel of $d^+_B$ and $g_Q$-orthogonal to the image of $d_B.$ This means that $\delta_B\alpha=0$. Moreover, using~(\ref{codiff_exp}), we have
\begin{equation*}
\delta_Bd_B\alpha=-\delta_B\,\bar\ast d_B\alpha=\bar\ast\left(d_B-\kappa_B\wedge\right)d_B\alpha=-\bar\ast\left(\kappa_B\wedge d_B\alpha\right).
\end{equation*}
In particular, when $\kappa_B=0,$ it follows that $d_B\alpha=\delta_B\alpha=0.$
Next, suppose that $\beta\in\Omega^+_B(M)$ is $g_Q$-orthogonal to the image of $d_B^+.$ Then $\delta_B\beta=0$ and
\begin{equation*}
d_B\beta=d_B\bar\ast\beta=\bar\ast\delta_B\beta+\kappa_B\wedge\bar\ast\beta.
\end{equation*}
Hence $(d_B-\kappa_B\wedge)\beta=0$. This implies that $\beta\in\mathcal{H}^+(\scrF)$ when $\kappa_B=0$.
\end{proof}

\section{Basic Dirac operator}\label{Sec_3}

%
%
%
%
%


Choose a point $x \in M$ and a $g_Q$-orthonormal basis $\{e_1,\cdots,e_m\}$ of $Q_x$. The \textit{Clifford algebra} of $Q_x$ is the complex algebra generated by $1$ and $e_1,\cdots,e_m$ subject to the relations
\[
e_i e_j +e_j e_i = -2\delta_{ij}.
\]
Define the {\it{transverse Clifford bundle}} $Cl(Q)$ over $M$ as the $\mathbb{Z}_2$-graded vector bundle
whose fiber at the point $x\in M$ is the Clifford algebra of $Q_x.$ This vector bundle is associated with the principal $SO(m)$-bundle of oriented $g_Q$-orthonormal frames in $Q.$  
The transverse Levi-Civita connection $\nabla^T$ induces on $Cl(Q)$ a leafwise flat connection $\nabla^{Cl(Q)}$ which
is compatible with the multiplication preserving the $\mathbb{Z}_2$-grading.

\begin{defn}
A principal bundle $P\rightarrow(M,\scrF)$ is called {\textit{foliated}}~\cite{Kam-Ton-2, Mol}
if it is equipped with a {\textit{lifted}} foliation $\scrF_P$ invariant under the structure group action, transversal to the
tangent space to the fiber and~$T\scrF_P$ projects isomorphically onto~$T\scrF.$ A vector bundle is foliated if its associated frame bundle is foliated.
\end{defn}

\begin{defn}
{\it{A transverse Clifford module}} $E$
is a complex vector bundle over $M$ equipped with an action of $Cl(Q)$, which we denote by $\bullet$
\[
Cl(Q) \otimes \Gamma(E) \xlongrightarrow{\bullet}\Gamma(E).  
\]
A transverse Clifford module $E$ is called {\it{self-adjoint}} if it is equipped
with a Hermitian metric $(\cdot,\cdot)$ such that the Clifford action is skew-adjoint at each point, i.e. 
\[
(X \bullet \psi_1, \psi_2) = -(\psi_1, X \bullet \psi_2),\quad \forall X \in\Gamma\left(Q\right),\,\forall\psi_1,\psi_2 \in \Gamma(E).
\]
\end{defn}
\begin{defn}
A {\it{Clifford connection}} on $E$ is a connection $\nabla^E$ which is compatible with the Clifford action, i.e.
\[
\nabla^E(X\bullet \psi) = (\nabla^{Cl(Q)}X) \bullet \psi + X \bullet (\nabla^E \psi),\quad\forall X\in \Gamma\left(Cl(Q)\right),\forall\psi \in \Gamma(E).
\]
\end{defn}

Let $E$ be a complex Hermitian {\textit{foliated}} bundle over $M$ such that
$E$ is a transverse Clifford module, self-adjoint and equipped with a Hermitian Clifford connection $\nabla^E.$ 
Moreover, we choose $\nabla^E$ to be a {\it{basic connection}} which means that
the connection and curvature forms of $\nabla^E$ are (Lie algebra-valued) basic forms (this is always possible by a result of~\cite{Kam-Ton-2}).
In this case, $(E,\nabla^E)$ is called a {\textit{basic Clifford module bundle}}.

Define the operator $\mathcal{D}: \Gamma(E) \longrightarrow \Gamma(E)$ between smooth sections of $E$ as the composition of the maps
\[
\Gamma(E)\xlongrightarrow{\nabla^E}\Gamma(Q)^\ast\otimes \Gamma(E) \cong\Gamma(Q)\otimes \Gamma(E)\xlongrightarrow{\bullet}\Gamma(E).
\]
In a local $g_Q$-orthonormal frame  $\{e_1,\cdots,e_m\}$ of $Q$ we write
\[
\mathcal{D}=\sum_{i=1}^m e_i\bullet\nabla^E_{e_i}.
\]
The operator $\mathcal{D}$ maps basic sections $\Gamma_B(E)=\{s\in\Gamma(E)\,|\,\nabla^E_Xs=0,\,\forall X\in\Gamma(T\scrF)\}$ to basic sections
but it is {\it{not}} formally self-adjoint as $D^*=D- \tau \bullet.$ Define the {\it{transverse Dirac operator}}~\cite{Dou-Gla-Kam-Yu,Gla-Kam,Pro-Ric-1} by
\[
\mathcal{D}_{tr}:=\mathcal{D}-\frac{1}{2}\tau\bullet.
\]
The operator $\mathcal{D}_{tr}$ is a self-adjoint, transversally elliptic differential operator but it is {\textit{not}} {{basic}}.
However, the operator 
\[
\mathcal{D}_B:=\mathcal{D}-\frac{1}{2}\tau_B\bullet,
\]
called {\it{basic Dirac operator}}~\cite{Hab-Ric}, is a basic self-adjoint transversally elliptic differential operator.
In particular, it is Fredholm. We emphasize that the basic Dirac operator $\mathcal{D}_B$ depends on the choice of the bundle-like metric $g$, i.e. \!not only on~$g_Q$ but also on the leafwise metric.
Nevertheless, it turns out that each choice of bundle-like metric associated to $g_Q$ produces a
new basic Dirac operator which is conjugate to the initial one. In particular, the dimension of the kernel
of the basic Dirac operator is the same for any bundle-like metric associated to $g_Q$~\cite{Hab-Ric}. 

\section{{Transverse \spinc-structures}}\label{Sec_4}
Suppose in the following that $Q$ is \spinc~\cite{Law-Mic}. In this case $(M,\scrF,g_Q)$ is said to be {\it{transversally}} \spinc~\cite{Pro-Ric}. We fix a {\it{transverse \spinc-structure}}
\[
\rho:Q \longrightarrow End(\mathcal{W}),
\]
where {{the complex Spinor bundle}} $\mathcal{W}$ is a basic Clifford module bundle.
The homomorphism $\rho$ is extended to an isomorphism of algebra bundles
\[
\rho:Cl^c(Q)\longrightarrow End(\mathcal{W}).
\]
Choose a $g_Q$-orthonormal basis $\{e_1,\cdots,e_m\}$ of $Q_x$. Then the chirality operator is the multiplication by  
\[
\left(\e\right)^k\,e_1\bullet\cdots\bullet e_m,
\] 
where $k=\frac{m}{2}$ if $m$ is even and $k=\frac{m+1}{2}$ if $m$ is odd. It is an involution and hence induces a natural splitting
\[
\mathcal{W}=\mathcal{W}^+\oplus\mathcal{W}^-,
\]
where $\mathcal{W}^\pm$ correspond to the $\pm$-eigenspaces. Moreover, the chirality operator is a {\it{basic bundle map}}~\cite{Pro-Ric},
i.e.~it preserves the space of basic spinors $\Gamma_B(\mathcal{W})$ of the complex spinor bundle $(\mathcal{W},\nabla^{\mathcal{W}})$. 

Consider a basic Hermitian connection $A$ on the (foliated) determinant line bundle $L$
of $\mathcal{W}$ i.e. $A$ is a Hermitian connection on $L$ which satisfies $\mathcal{L}_{X}A=\iota_X A=0,\,\forall X\in\Gamma(T\scrF)$.
Together with the transverse Levi-Civita connection $\nabla^T$, $A$ induces a basic Clifford Hermitian connection $\nabla^{A}$ on $\mathcal{W}$.
Let $\mathcal{D}_B^A$ be the basic Dirac operator
associated to $\nabla^A.$ Then we get the following Weitzenb\"{o}ck formula (see~\cite{Gla-Kam}):

\begin{prop}
Let $(M,\scrF,g_Q)$ and $(\mathcal{W},\nabla^A)$ be as above. Then 
\begin{equation}\label{weitzenbock}
(\mathcal{D}_B^A)^2\psi=\left(\nabla^{A}_{tr}\right)^{\ast}\nabla^{A}_{tr}\psi+\frac{1}{4}\left(s^T+ |\kappa_B|^2 - 2\delta_B\kappa_B\right)\psi+\frac{1}{2}F_A \bullet\psi,
\end{equation}
for any basic spinor $\psi \in \Gamma_B(\mathcal{W})$ 
Here~$s^T$ is the transverse scalar curvature, $F_A$ is the curvature of
the connection~$A$ and 
\begin{equation}\label{nabla_transverse}
\left(\nabla^{A}_{tr}\right)^{\ast}\nabla^{A}_{tr}\psi=-\sum_{i=1}^m\nabla^A_{e_i}\nabla^A_{e_i}\psi+\sum_{i=1}^m\nabla^A_{\nabla^T_{e_i}e_i}\psi+\nabla^A_{\tau_B},
\end{equation}
where $\{e_1,\cdots,e_m\}$ is a local basic $g_Q$-orthonormal frame of $Q.$ and $\nabla^{A}_{tr}$ is the operator defined by $\displaystyle \nabla^{A}_{tr}\psi=\sum_{i=1}^me_i^{\flat_{g_Q}} \otimes(\nabla_{e_i}^A\psi).$
\end{prop}

\begin{rem}
Slesar introduced in~\cite{Sle} a {\it{modified connection}} on the space $ \Gamma_B(\mathcal{W})$ given by
\begin{equation*}
\bar{\nabla}^A_{X}\psi:=\nabla^A_{X}\psi-\frac{1}{2}g(X,\tau_B)\psi,
\end{equation*}
where $X\in\Gamma(TM)$. Using the above modified connection, the Weitzenb\"{o}ck formula becomes
\begin{equation}\label{weitzenbock-2}
(\mathcal{D}_B^A)^2\psi=\left({\bar{\nabla}^A_{tr}}\right)^{\ast}\bar{\nabla}^A_{tr}\psi+\frac{1}{4}s^T\psi+\frac{1}{2}F_A \bullet\psi,
\end{equation}
where $\displaystyle{\bar{\nabla}^A_{tr}}\psi=\sum_{i=1}^me_i^{\flat_{g_Q}} \otimes(\bar{\nabla}_{e_i}^A\psi)$ for a local basic $g_Q$-orthonormal frame $\{e_1,\cdots,e_m\}$ of $Q.$
\end{rem}

\section{Sobolev spaces and basic elliptic theory}\label{Sec_6}

Let $E$ be a foliated vector bundle on $M$ equipped with a basic Hermitian structure and a compatible basic connection $\nabla^E$.  
Denote by ${\mathcal E}_B$ the sheaf of germs of smooth basic sections of $E$ and by $\Gamma_B(E)$ the space of the smooth basic sections of $E$. Thus, $\Gamma_B(E)=\Gamma({\mathcal E}_B)$.  
For $p\in [1,\infty)$, we denote by $\|\cdot\|_{L^{p}}$ the $L^p$-norm on $\Gamma(E)$ and, for any integer $k\geqslant 0$ and $p\in [1,\infty)$, we denote by $\|\cdot\|_{k,p}$ the Sobolev norm on $\Gamma(E)$:
\[
\|u\|_{k,p}=\sum_{j=0}^k\left(\int_M |(\nabla^E)^ju|^pv_g\right)^{1/p}, \quad u\in \Gamma(E),
\]
where $v_g$ stands for the Riemannian volume form.  
Let $L^{p}\left(\Gamma_B(E)\right)$ be the $L^p$-norm closure of the space $\Gamma_B(E)$ of the smooth basic sections of $E$. Let $W^{k,p}\left(\Gamma_B(E)\right)$ be the closure of $\Gamma_B(E)$ under the Sobolev norm $\|\cdot\|_{k,p}$.  
To ease the notations, we let $W^{k,p}$ stand for $W^{k,p}\left(\Gamma_B(E)\right)$ unless it is necessary to specify the space.

Consider a foliated chart $\Omega\subset M \cong U\times V$ with some open subsets $U\subset {\mathbb R}^d$ and $V\subset {\mathbb R}^m$. 
We assume that $U$ and $V$ are connected. 
We will also assume that the chart is regular, which means that it can be extended to a foliated chart $\Omega^\prime\cong U^\prime\times V^\prime$ such that  $\overline{\Omega}\subset \Omega^\prime$, $\overline{U} \subset U^\prime\subset {\mathbb R}^d$ and $\overline{V}\subset V^\prime \subset {\mathbb R}^m$. There exists a local trivialization $E\left|_\Omega\right. \cong U\times V \times \mathbb C^N$ of $E$ over $\Omega$ (a distinguished local trivialization, see, for instance, \cite{Mol}) such that any basic section $s\in \Gamma(\Omega, {\mathcal E}_B)$ on $\Omega$ is written as 
\begin{equation}\label{local-form}
s(x,y)=u(y), \quad x\in U,\quad y\in V,
\end{equation}
with some vector-valued function $u\in C^\infty(V,\mathbb C^N)$. It can be shown that~$s$ belongs to $W^{k,p}(\Gamma(\Omega, {\mathcal E}_B))$ if and only if $u\in W^{k,p}(V,\mathbb C^N)$. Moreover, the $W^{k,p}(\Gamma(\Omega, {\mathcal E}_B))$-norm of $s$ is equivalent to its $W^{k,p}( U\times V, \mathbb C^N)$-norm which, in turn, equals the $W^{k,p}(V, \mathbb C^N)$-norm of $u$:
\begin{equation}\label{equiv-norms}
C_1\|u\|_{W^{k,p}(V, \mathbb C^N)} \leqslant \|s\|_{W^{k,p}(\Gamma_B(E\left|_\Omega\right.))} \leqslant C_2 \|u\|_{W^{k,p}(V, \mathbb C^N)}
\end{equation}
with some constants $C_1, C_2>0$, independent of $s$. In this way, we obtain isomorphisms 
\[
\Gamma(\Omega, {\mathcal E}_B)\cong C^\infty(V,\mathbb C^N), \quad W^{k,p}(\Gamma(\Omega, {\mathcal E}_B))\cong W^{k,p}(V, \mathbb C^N).
 \]

The basic Sobolev spaces $W^{k,p}$ satisfy the {\textit{basic Sobolev embedding theorem}} (see~\cite[Lemma 3.7]{Bel-Par-Ric} when $p=2$). 

\begin{thm}\label{basic_sobolev_embedding}
Given integers $l$, $q$ such that $0\leqslant l\leqslant k$ and $l - \frac{m}{q}\leqslant k-\frac{m}{p}<l$  with~$m$ the codimension of the foliation~$\scrF$, there is a continuous inclusion 
\[
W^{k,p}\hookrightarrow W^{l,q}.
\]
Moreover, for $l < k-\frac{m}{p}$
\[
W^{k,p}\hookrightarrow C^l,
\]
where $C^l$ stands for $C^l(\Gamma_B(E))$ the space of basic sections of $E$ of class $C^l$.
\end{thm}

\begin{proof}
The theorem immediately follows from \eqref{equiv-norms} and the classical Sobolev embedding theorem. 
\end{proof}

We also get a {\textit{basic Rellich theorem}} (cf.~\cite[Proposition 4.5]{Kam-Ton-1}).

\begin{thm}
The inclusion
\[
W^{k+1,p}\hookrightarrow W^{k,p}
\]
is compact for all $k$ and $p$.
\end{thm}

\begin{proof}
Since $W^{k,p}$ is closed in $W^{k,p}\left(\Gamma(E)\right)$, this theorem is an easy consequence of the classical Rellich theorem. 
\end{proof}

Finally, we get a {\textit{basic Sobolev multiplication theorem}} (the standard proof~\cite[Theorem 4.39]{Ada-Fou} works in the basic case).

\begin{thm}\label{basic_sobolev_multiplication}
When $0 \leqslant l \leqslant k$:
\begin{enumerate}
\item{Above the borderline: if $k-\frac{m}{p}>0$ and $k-\frac{m}{p}\geqslant l-\frac{m}{q},$ then we have a continuous map
$$W^{k,p}\times W^{l,q}\longrightarrow W^{l,q}.$$
In the particular case $k=l$ and $p=q$: if $k- \frac{m}{p} >0,$ then $$W^{k,p}\times W^{k,p}\longrightarrow W^{k,p}.$$}
\item{Under the borderline: if $k-\frac{m}{p}<0$ and $l-\frac{m}{q}<0,$ then we have a continuous map $$W^{k,p}\times W^{l,q}\longrightarrow W^{t,r},$$
with $0\leqslant t \leqslant l$ and $r$ such that $0<\frac{t}{m}+\frac{1}{p}-\frac{k}{m}+\frac{1}{q}-\frac{l}{m}\leqslant \frac{1}{r}\leqslant 1.$}
\item{On the borderline: if $k-\frac{m}{p} = 0$ and $l- \frac{m}{q} \leqslant 0,$ then we have a continuous map
$$\left(W^{k,p} \cap L^\infty \right) \times \left(W^{l,q}  \cap L^\infty \right) \longrightarrow \left(W^{l,q} \cap L^\infty \right).$$
If $l-\frac{m}{q} < 0,$ then we can improve this to a continuous map
$$\left(W^{k,p} \cap L^\infty \right) \times W^{l,q}  \longrightarrow W^{l,q}.$$
}
\end{enumerate}
\end{thm}

We can define $\mathcal{G}^{k+1,p}_B=W^{k+1,p}\left(\Gamma_B\left(\mathrm{End}\left(\mathcal{W}\right)\right)\right)\cap C^0_B(M,S^1)$ whenever $k+1>\frac{m}{p}$. 
It can be shown that $\mathcal{G}^{k+1,p}_B$ is an infinite-dimensional Lie group. 

\quad

Let $E_1$, $E_2$ be foliated vector bundles on $M$ equipped with basic Hermitian structures and compatible basic connections. 
Recall (see, for instance, \cite{Kac}) that a basic differential operator of order $\ell$ from $E_1$ to $E_2$ is a morphism of sheaves $P: {\mathcal E}_{1,B}\to {\mathcal E}_{2,B}$ such that, in any foliated chart $\Omega$ with coordinates $(x,y)\in U\times V \subset {\mathbb R}^d \times  {\mathbb R}^m$ equipped with distinguished local trivializations of $E_1$ and $E_2$ over it, the associated map $P_\Omega: \Gamma(\Omega,{\mathcal E}_{1,B})\cong C^\infty(V,\mathbb C^N)\to \Gamma(\Omega,{\mathcal E}_{2,B})\cong C^\infty(V,\mathbb C^N)$ has an expression 
\[
P_\Omega=\sum_{|\alpha|\leqslant \ell}a_\alpha(y)\frac{\partial^{|\alpha| }}{\partial y_1^{\alpha_1}\cdots \partial y_m^{\alpha_m}},
 \] 
 where $\alpha=(\alpha_1,\ldots,\alpha_m)\in {\mathbb N}^m$ is a multi-index, and, for  any $\alpha$ with $|\alpha|\leqslant \ell$, $a_\alpha$ is a smooth $N_1\times N_2$-matrix valued function on $V$.  

It is easy to see that any basic differential operator $P$ from $E_1$ to $E_2$ of order~$\ell$ extends to a continuous map from $W^{k+\ell,p}(\Gamma_B(E_1))$ to $W^{k,p}(\Gamma_B(E_2))$ for any integer $k\geqslant 0$ and $p\in (1,\infty)$. 

A basic differential operator $P$ from $E_1$ to $E_2$ is said to be transversally elliptic if, in any foliated chart $\Omega\cong U\times V \subset {\mathbb R}^d \times  {\mathbb R}^m$ equipped with distinguished local trivializations of $E_1$ and $E_2$ over it, the associated operator $P_\Omega: C^\infty(V,\mathbb C^N)\to  C^\infty(V,\mathbb C^N)$ is an elliptic differential operator (i.e. its principal symbol $\sigma_P(y,\eta)$ is invertible for any $(y,\eta)\in V\times ({\mathbb R}^q\setminus \{0\})$).

We have the following regularity result ($L^p$-estimates) for basic transversally elliptic operators which will be crucial for the compactness of the moduli space.

 \begin{thm}\label{regularity}
Let $P:\Gamma_B(E_1)\rightarrow\Gamma_B(E_2)$ be a basic transversally elliptic operator of order $\ell$. Then, for every integer $k\geqslant 0$ and $p\in (1,\infty)$, there are constants $c_1>0, c_2>0$ (depending on $k$ and $p$) such that  
\begin{equation}\label{eq:regularity}
\|s\|_{k+\ell,p}\leqslant c_1\|P s\|_{k,p}+c_2\|s\|_{k,p}, \quad  s\in W^{k+\ell,p}(\Gamma_B(E_1)).
\end{equation}
Moreover, if $p=2$, $\ell=1$ and $s$ is $L^2$-orthogonal to the kernel of $P$ then one can assume $c_2=0$.
\end{thm}

\begin{proof}
Let $M=\cup_{i=1}^r \Omega_i$ be a finite covering of $M$ by foliated charts $\Omega_i\cong U_i\times V_i \subset {\mathbb R}^d\times {\mathbb R}^m$ equipped with distinguished local trivializations of the vector bundles $E_1$ and $E_2$ over it. 
Write the restriction $s_i$ of $s$ to $\Omega_i$ in the form \eqref{local-form} with some $v_i\in W^{k,p}(V_i, \mathbb C^N)$.
Using \eqref{equiv-norms}, we get
\[
\|s_i\|_{W^{k+\ell,p}(\Gamma(\Omega_i,{\mathcal E}_{1,B}))} \leqslant C_2 \|v_i\|_{W^{k+\ell,p}(V_i, \mathbb C^N)}.
\]
Take an open subset $V^{\prime\prime}_i\subset V^{\prime}_i$ such that $\overline{V_i}\subset V^{\prime\prime}_i$.
By classical $L^p$-estimate for the elliptic operator $P$, we have 
\[
\|v_i\|_{W^{k+\ell,p}(V_i, \mathbb C^N)}\leqslant c_1\|P_{\Omega_i} v_i\|_{W^{k,p}(V^{\prime\prime}_i, \mathbb C^N)}+c_2\|v_i\|_{W^{k,p}(V^{\prime\prime}_i, \mathbb C^N)}.
\] 
Using \eqref{equiv-norms} again, we get 
\[
\|s_i\|_{W^{k+\ell,p}(\Gamma(\Omega_i,{\mathcal E}_{1,B}))}\leqslant c_1\|Ps_i\|_{W^{k,p}(\Gamma(\Omega^{\prime\prime}_i,{\mathcal E}_{1,B}))}+c_2\|s_i\|_{W^{k,p}(\Gamma(\Omega^{\prime\prime}_i,{\mathcal E}_{1,B}))}.
\] 

Finally, we have 
\begin{align*}
\|s\|_{k+\ell,p}\leqslant &  \sum_{i=1}^r \|s_i\|_{W^{k+\ell,p}(\Gamma(\Omega_i,{\mathcal E}_{1,B}))}\\
\leqslant & c_1 \sum_{i=1}^r \|Ps_i\|_{W^{k,p}(\Gamma(\Omega^{\prime\prime}_i,{\mathcal E}_{1,B}))}+c_2 \sum_{i=1}^r\|s_i\|_{W^{k,p}(\Gamma(\Omega^{\prime\prime}_i,{\mathcal E}_{1,B}))}
\\
\leqslant & rc_1 \sum_{i=1}^r \|Ps\|_{k,p}+rc_2 \sum_{i=1}^r\|s\|_{k,p}, 
\end{align*}
that completes the proof of \eqref{eq:regularity}.

Now we fix $k$ and assume that $p=2$ and $\ell=1$. Since the $L^2$-spectra of the operators $P$ and $P^*$ are discrete, for any $s\in W^{2,2}(\Gamma_B(E_1))$, $L^2$-orthogonal to $\ker P=\ker P^*P$, we have
\[
\|Ps\|^2_{L^2}=(P^*Ps,s)\geqslant \varepsilon\|s\|^2_{L^2},
\] 
and, for any $s\in W^{2,2}(\Gamma_B(E_1))$, $L^2$-orthogonal to $\ker P^*=\ker PP^*$, we have
\[
\|P^*s\|^2_{L^2}=(PP^*s,s)\geqslant \varepsilon\|s\|^2_{L^2},
\] 
where $\varepsilon >0$ is independent of $s$. Here $(\cdot,\cdot)$ denotes the inner product in $L^2$.

Using these estimates, for any $s\in W^{2k,2}(\Gamma_B(E_1))$, $L^2$-orthogonal to the kernel of $P$, we get
\begin{equation}\label{estimate}
((P^*P)^ks,s)\geqslant \varepsilon^k\|s\|^2_{L^2},
 \end{equation}
and, for any $s\in W^{2k,2}(\Gamma_B(E_1))$, $L^2$-orthogonal to the kernel of $P^*$, we get
\begin{equation}\label{estimate*}
((PP^*)^ks,s)\geqslant \varepsilon^k\|s\|^2_{L^2}.
 \end{equation}

It follows by \eqref{estimate} and \eqref{estimate*} that, for any $s\in W^{2k,2}(\Gamma_B(E_1))$, $L^2$-orthogonal to the kernel of $P$, we obtain that 
\begin{equation}\label{estimate-k}
\|s\|^2_{k,2}\leqslant c_3((P^*P)^ks,s),
\end{equation}
with some constant $c_3>0$, independent of $s$.
Indeed, if $k=2K$ is even, we can write 
\[
((P^*P)^ks,s)=\|(P^*P)^Ks\|^2_{L^2},
\]
while if $k=2K+1$ is odd, 
\[
((P^*P)^ks,s)=\|(PP^*)^KPs\|^2_{L^2}.
\]
  
Finally, for any $s\in W^{2k+1,2}(\Gamma_B(E_1))$, $L^2$-orthogonal to the kernel of $P$, using continuity of the operator $(PP^*)^k : W^{2k,2}\to L^2$ and  \eqref{estimate-k}, we derive 
\[
\|Ps\|^2_{2k,2}\geqslant c_4\|(PP^*)^kPs\|^2_{L^2}=((P^*P)^{2k+1}s,s)\geqslant c_3^{-1}c_4\|s\|^2_{2k+1,2}.
\]
Similarly, for any $s\in W^{2k+2,2}(\Gamma_B(E_1))$, $L^2$-orthogonal to the kernel of $P$, using continuity of the operator $P^*(PP^*)^k : W^{2k+1,2}\to L^2$ and  \eqref{estimate-k}, we derive 
\[
\|Ps\|^2_{2k+1,2}\geqslant c_5\|P^*(PP^*)^kPs\|^2_{L^2}=((P^*P)^{2k+2}s,s)\geqslant c_3^{-1}c_5\|s\|^2_{2k+2,2},
\]
that completes the proof. 
\end{proof}

Later on, when dealing with transversality issues of the moduli space, we will need a unique continuation property for the basic Dirac operator $\mathcal{D}_B$ (see ~\cite[proof of Theorem 5.2]{Wan} for a similar fact for the basic Laplacian).


\begin{thm}\label{continuation_thm}
Let $M$ be a connected compact Riemannian manifold endowed with a {{Riemannian foliation}}.
Then the basic Dirac operator $\mathcal{D}_B$ satisfies the unique continuation property: if $\psi\in \Gamma_B\left(\mathcal{W}\right)$ such that $\mathcal{D}_B\psi=0$ and $\psi=0$ on an open subset $S\subset M$, then $\psi\equiv 0$.  
\end{thm}

\begin{proof}
Without loss of generality, we can assume that $S$ is a saturated subset. Take a foliated chart $\Omega\cong U\times V\subset {\mathbb R}^d \times {\mathbb R}^m$ equipped with distinguished local trivialization of  $\mathcal{W}$ over it. Then the intersection $S$ with $\Omega$ has the form $\Omega\cap S \cong U\times W$, where $W$ is an open subset in $V$. Let us write $\psi$ as in \eqref{local-form}
with some $v\in C^\infty(V,{\mathcal C}^{\dim {\mathcal W}})$.  It is easy to see that the operator $(\mathcal{D}_B)_\Omega : C^\infty(V, {\mathcal C}^{\dim {\mathcal W}})\to C^\infty(V, {\mathcal C}^{\dim {\mathcal W}})$  is a Dirac type operator, and we have 
$(\mathcal{D}_B)_\Omega v=0$ and $v=0$ on~$W$. Therefore, the desired statement immediately follows from the unique continuation property for Dirac type operators (see, for instance, \cite{booss,booss-woj}). 
\end{proof}

\section {Basic Seiberg--Witten equations}\label{Sec_5}

Suppose that $M$ is a closed oriented smooth manifold of dimension $n$ endowed with a {{Riemannian foliation}} $\scrF$
of codimension $4$ and $g$ a {{bundle-like}} metric on~$M$. 
Suppose furthermore that the normal bundle $Q$ is \spinc~and fix a {{transverse \spinc-structure}}
$\rho:Cl^c(Q)\longrightarrow End(\mathcal{W}),$
where $\mathcal{W}$ is the complex spinor bundle. Let~$L$ be the determinant line bundle of $\mathcal{W}.$
We denote by $\mathcal{K}_B(L)$ 
the space of (smooth) Hermitian basic connections on $L.$
The space of basic connections on $L$ is an affine space modeled on the space of sections of $C^\infty(\wedge^1Q^\ast,\mathbb{R})$ 
which are constant along the leaves. Hence, it can be identified with the space $\Omega^1_B(M).$  


We define the {\textit{basic Seiberg--Witten equations}} by
\begin{equation}\label{SW_equations}
\left\{\begin{aligned}
\mathcal{D}^A_B(\psi)&=0,\\
F_A^+&=\sigma(\psi),
\end{aligned}\right.
\end{equation}
for the couple $(A,\psi)$ with $A\in\mathcal{K}_B(L)$ 
and $\psi\in\Gamma_B(\mathcal{W}^+)$ a basic section of $\Gamma(\mathcal{W}^+)$. 
Here, $\mathcal{D}_B^A$ is the basic Dirac operator associated to the connection $\nabla^A$ on $\mathcal{W}$ induced by 
the transverse Levi-Civita connection $\nabla^T$ and $A$;
$F_A^+$
is the $\bar\ast$-self-dual part of the curvature of $A$; and
 $\sigma:\Gamma_B(\mathcal{W}^+)\longrightarrow\sqrt{-1}\Omega^+_B(M)$ is the standard quadratic map given by
$\sigma(\psi)=\psi\sqrt{-1}\,\overline{\psi}$, where $\psi\sqrt{-1}\,\overline{\psi}$ corresponds to a traceless endomorphism of $\mathcal{W}^+$
which we can identify to a form in $\Omega^+_B(M)\otimes \mathbb{C}.$

%

The {\textit{perturbed basic Seiberg--Witten equations}}
additionally depend on a basic self-dual 2-form $\mu\in\Omega^+_B(M)$:
\begin{equation}\label{perturbed_SW}
\left\{\begin{aligned}
\mathcal{D}^A_B(\psi)&=0,\\
F_A^+&=\sigma(\psi)+\sqrt{-1}\mu.
\end{aligned}\right.
\end{equation}
%




The {\it{basic gauge group}} $\mathcal{G}_B$ consists of gauge transformations $C^\infty(M,U(1))$ of the $U(1)$-principle (foliated) bundle associated to $L$
leaving invariant the lift $\widetilde\scrF$ of the foliation $\scrF$ to the $U(1)$-principle bundle associated to $L.$
Hence, the group $\mathcal{G}_B$ is identified with the (smooth) basic maps $C^\infty_B(M,S^1)$. An element $u\in\mathcal{G}_B$
acts on a pair $(A,\psi)\in\mathcal{K}_B(L)\times\Gamma_B(\mathcal{W}^+)$ by $(u^\ast A,u\psi):=(A-2u^{-1}du,u\psi).$
The basic gauge group $\mathcal{G}_B$ fixes the space $\mathcal{K}_B(L)$ 
and it is easy to see that the (perturbed) basic Seiberg--Witten equations are invariant under the action of $\mathcal{G}_B$.
We will also consider the based basic gauge group $\mathcal{G}_0:=\{u\in\mathcal{G}\,|\,u(x_0)=1\}$ where a point $x_0 \in M$ is fixed.
The action of $\mathcal{G}_B$ is free except at pairs $(A,0)$ with stabilizer given by constant maps.
On the other hand, the action of $\mathcal{G}_0$ is always free. The group $\mathcal{G}_B$ is actually the product
of $\mathcal{G}_0$ with constant maps.

For technical reasons, we need to consider Banach spaces or Hilbert spaces. So, we define the {\textit{basic Seiberg--Witten map}} 
\begin{equation}\label{SW_map}
\begin{array}{cccc}
\mathfrak{sw}_{\rho,\mu}:&W^{2,2}\left(\mathcal{K}_B(L)\right)\times W^{2,2}\left(\Gamma_B(\mathcal{W}^+)\right)&\longrightarrow&W^{1,2}\left( \sqrt{-1}\Omega^+_B(M)\right)\times W^{1,2}\left(\Gamma_B(\mathcal{W}^-)\right)\\
&(A,\psi)&\longmapsto&\left(F_A^+-\sigma(\psi)-\sqrt{-1}\mu,\mathcal{D}^A_B(\psi)\right).
\end{array}
\end{equation}
We denote 
$\mathfrak{m}_{\rho,\mu}:=\{(A,\psi)\,|\,\mathfrak{sw}_{\rho,\mu}(A,\psi)=0\}$ and we define the moduli spaces
\begin{equation}\label{moduli_space}
\mathcal{M}_{\rho,\mu}:=\frac{\mathfrak{m}_{\rho,\mu}}{\mathcal{G}^{3,2}},\quad\overline{\mathcal{M}}_{\rho,\mu}:=\frac{\mathfrak{m}_{\rho,\mu}}{\mathcal{G}_0^{3,2}}.
\end{equation}
The space $\mathcal{M}_{\rho,\mu}$ is the quotient of $\overline{\mathcal{M}}_{\rho,\mu}$ by the action of constant maps.

%
%

\section{Compactness of the moduli space}\label{Sec_7}

The aim of this section is to show compactness of the moduli space. 
We follow closely the excellent exposition in~\cite{Li}.
Throughout this section all the constants depend on the metric $g$, the perturbation $\mu$ in~(\ref{perturbed_SW}) and the transverse \spinc-structure.
\subsection{Pointwise estimates of $F_A^+$ and $\psi$}

We shall first prove an a priori pointwise bound on the basic spinor $\psi$ and on the curvature $F_A^+$. 
The next lemma follows from a simple computation.

\begin{lem}\label{laplacian_expr}
Let $f$ be a basic function 
and $\{e_1,\cdots,e_m\}$ be a local $g_Q$-orthonormal basic frame of $Q$.
Then
\begin{equation*}
\Delta_Bf=-\sum_{i=1}^m e^2_if+\left(\sum_{i=1}^m\nabla^T_{e_i}e_i\right)f+\tau_Bf.
\end{equation*}
\end{lem}
We deduce a basic version of {\textit{{Kato's inequality}}.

\begin{lem}\label{basic_kato}
Let  $\psi \in \Gamma_B(\mathcal{W})$. 
Then 
\begin{equation*}
\Delta_B\left(|\psi|^2\right)\leqslant 2\,\mathrm{ Re}\left(\left(\nabla^{A}_{tr}\right)^{\ast}\nabla^{A}_{tr}\psi,\psi\right).
\end{equation*}
\end{lem}
\begin{proof}
Apply Lemma~\ref{laplacian_expr} to the basic function $|\psi|^2:$ 
\begin{eqnarray*}
\Delta_B\left(|\psi|^2\right)&=&-\sum_{i=1}^m e^2_i(\psi,\psi)+\left(\sum_{i=1}^m\nabla^T_{e_i}e_i\right)(\psi,\psi)+\tau_B(\psi,\psi)\\
&=&\sum_{i=1}^m\left(-2(\nabla^A_{e_i}\psi,\nabla^A_{e_i}\psi)-(\nabla^A_{e_i}\nabla^A_{e_i}\psi,\psi)-(\psi,\nabla^A_{e_i}\nabla^A_{e_i}\psi)\right)\\
&&+\sum_{i=1}^m(\nabla^A_{\nabla^T_{e_i}e_i}\psi,\psi)+(\psi,\nabla^A_{\nabla^T_{e_i}e_i}\psi)+(\nabla^A_{\tau_B}\psi,\psi)+(\psi,\nabla^A_{\tau_B}\psi)\\
&=&-2|\nabla^{A}_{tr}\psi|^2+2\,\mathrm{ Re}\left(\left(\nabla^{A}_{tr}\right)^{\ast}\nabla^{A}_{tr}\psi,\psi\right).
\end{eqnarray*}
In the last line we used expression (\ref{nabla_transverse}).
\end{proof}

We can now combine the basic Kato inequality with the Weitzenb\"{o}ck formula to get the desired {{a priori pointwise estimates}}.
\begin{lem}\label{pointwise_bound}
Let  $(A,\psi)\in \mathfrak{m}_{\rho,\mu}.$ 
Then, at every point of $M$ we have
\begin{equation*}
|\psi|^2\leqslant \sup_{M}\left(-s^T-|\kappa_B|^2+2\delta_B\kappa_B+4|\mu|,0\right).
\end{equation*} 
\end{lem}
\begin{proof}
We apply Lemma~\ref{basic_kato} and the Weitzenb\"ock formula~(\ref{weitzenbock}) to $(A,\psi)\in \mathfrak{m}_{\rho,\mu}$:
\begin{eqnarray*}
\Delta_B\left(|\psi|^2\right)&\leqslant &2\,\mathrm{ Re}\left(\left(\nabla^{A}_{tr}\right)^{\ast}\nabla^{A}_{tr}\psi,\psi\right)\\
&\leqslant &-\frac{s^T}{2}|\psi|^2-\frac{1}{2}|\kappa_B|^2|\psi|^2+\delta_B\kappa_B|\psi|^2-\mathrm{ Re}(F_A\bullet\psi,\psi)\\
&\leqslant &-\frac{s^T}{2}|\psi|^2-\frac{1}{2}|\kappa_B|^2|\psi|^2+\delta_B\kappa_B|\psi|^2-\mathrm{ Re}((\sigma(\psi)+\e \mu)\bullet\psi,\psi)\\
&\leqslant &-\frac{s^T}{2}|\psi|^2-\frac{1}{2}|\kappa_B|^2|\psi|^2+\delta_B\kappa_B|\psi|^2-\frac{1}{2}|\psi|^4+2|\mu||\psi|^2.
\end{eqnarray*}
Let $x_0$ be a point where $|\psi|^2$ is maximal. At the point $x_0$, we have $\tau\left(|\psi|^2\right)|_{x_0}=0$
and thus $\left(\Delta^gf\right)(x_0)=\left(\Delta_Bf\right)(x_0)$ for any basic function $f$,
where $\Delta^g$ is the (usual) Riemannian Laplacian with respect to the metric $g$. Then,
either, $|\psi|^2> 0$ at $x_0$ and
\begin{equation*}
0\leqslant \left(-\frac{s^T}{2}-\frac{1}{2}|\kappa_B|^2+\delta_B\kappa_B-\frac{1}{2}|\psi|^2+2|\mu|\right)(x_0)|\psi|^2(x_0),
\end{equation*}
or the spinor $\psi$ is identically zero.
\end{proof}

\begin{rem}
When $\mu=0$, $\displaystyle |F_A^+(x)|=\frac{|\psi(x)|^2}{2}$, we have then a pointwise estimate of $F_A^+.$ We get similarly a pointwise estimate of $F_A^+$ when $\mu\neq 0$.
\end{rem}


\subsection{Basic gauge fixing}

We would like to show that the pointwise bounds on the curvature $F_A^+$ give $W^{2,2}$-bounds on the connection $A$ if written in an appropriate gauge. In a first step, the following Lemma establishes $L^2$-bounds on the derivatives of both $F_A^+$ and $\psi$.

\begin{lem}\label{derivatives-bound}
If $(A,\psi)\in \mathfrak{m}_{\rho,\mu}$, then $\nabla^A_{tr} \psi$, $\nabla^T F_A^+$ and $dF_A^+$ are all bounded in $L^2$.
\end{lem}

\begin{proof}
If we take the inner product with $\psi$ and integrate, we deduce from the Weitzenb\"{o}ck formula (\ref{weitzenbock}) that
$$\|\nabla^A_{tr} \psi \|^2_{L^2} \leqslant c_1 \|\psi \|^2_{L^2} + c_2\|\psi \|^2_{L^4},$$
for some constants $c_1$ and $c_2$. By Lemma~\ref{pointwise_bound}, 
we get that $\nabla^A_{tr} \psi$ is bounded in $L^2$. Next, we use the curvature equation to write
$$\nabla^T F_A^+ = \nabla^A(\sigma(\psi)+\sqrt{-1} \mu)=(\nabla^A \psi)\sqrt{-1}~\overline{\psi} + \psi \sqrt{-1}(\nabla^A\overline{ \psi}) + \sqrt{-1} \nabla^T \mu.$$
The last term $\mu$ is smooth and the first two terms are in $L^2$ because $\psi$ is in $C^0$ and $\nabla^A \psi$ is in $L^2$.
Finally, 
$$dF_A^+ = d_B F_A^+ = \sum_{i=1}^me_i^{\flat_{g_Q}}\wedge\nabla^T_{e_i} F_A^+.$$
Hence $dF_A^+$ is just a projection of $\nabla^T F_A^+$ and thus also bounded in $L^2$. 
\end{proof}

We next deduce a $W^{1,2}$-bound on the curvature $F_A^+$. For this purpose, fix a smooth 
reference connection $A_0\in W^{2,2}\left(\mathcal{K}_B(L)\right)$ and we write $A=A_0+\sqrt{-1}\alpha.$

\begin{lem}\label{curvcurv}
There exist constants $c_1$, $c_2$ and $c_3$ depending on $A_0$ such that
$$\|F_A^+\|_{1,2} \leqslant  c_1 \|dF_A^+\|_{L^2} + c_2 \|F_A^+\|_{L^2} + c_3.$$
\end{lem}

\begin{proof}
By basic Hodge decomposition, we can write $$F_A^+ - F_{A_0}^+ = \alpha'_H+\beta',$$
where $\alpha'_H$ is the $\Delta_B$-harmonic part and $\beta'$ is $L^2$-orthogonal to the $\Delta_B$-harmonic forms. Note that $\alpha'_H$ and $\beta'$ are not necessarily $\bar\ast$-self-dual.
We  apply Theorem~\ref{regularity} to get
$$\|\alpha_H' \|_{1,2} \leqslant c_1\|\alpha_H' \|_{L^2} \leqslant c_1\|F_A^+ \|_{L^2} + c_1\|F_{A_0}^+ \|_{L^2}.$$
On the other hand, 
\begin{align*}
\|\beta'\|_{1,2}&\leqslant  c_2\|(d_B+\delta_B)\beta'\|_{L^2}\\
&\leqslant  c_2\|(d_B-\bar\ast d_B+\iota_{\tau_B})(F_A^+-F_{A_0}^+)\|_{L^2}\\
&\leqslant 2c_2 \|dF_A^+\|_{L^2}+2c_2\|dF_{A_0}^+\|_{L^2}+c_3\|F_A^+-F_{A_0}^+ \|_{L^2}.
\end{align*}
Hence,
\begin{eqnarray*}
\|F_A^+\|_{1,2}&\leqslant & \|F_{A_0}\|_{1,2}+\|\alpha'_H\|_{1,2}+\|\beta'\|_{1,2}\\
&\leqslant &  \|F_{A_0}\|_{1,2}+c_1 \|F_A^+\|_{L^2}+c_1\|F_{A_0}^+\|_{L^2}+2c_2 \|dF_A^+\|_{L^2}+2c_2\|dF_{A_0}^+\|_{L^2}\\
& & + c_3\|F_A^+-F_{A_0}^+ \|_{L^2}.
\end{eqnarray*}
\end{proof}

The next Lemma is an instance of Uhlenbeck's gauge fixing Lemma, see for example~\cite{Mor}. In particular, we can apply it to deduce the desired $W^{2,2}$-bound on~$\alpha$ from the $W^{1,2}$-bound on $F_A^+$.
Unfortunately, we need to add the hypothesis that the foliation $\scrF$ is taut so that $\ker\left(\delta_B \oplus d_B^+\right)=H^1(\scrF)$ and that $H^1(\scrF)\cap H^{1}(M,\mathbb{Z})$
is a lattice in $H^1(\scrF)$ in order to get the desired estimate.
The latter condition is satisfied, for instance, on any closed {\it K-contact} manifold (see Section~\ref{Sec_11}). Indeed, in that case, $H^1(\scrF)$ is isomorphic to
$H^{1}(M,\mathbb{R})$~\cite[Proposition 7.2.3]{bo-ga}. 


\begin{lem}\label{gauge_fixing}
Suppose that the foliation $\scrF$ is taut and $H^1(\scrF)\cap H^{1}(M,\mathbb{Z})$
is a lattice in $H^1(\scrF)$.
Then, for any $k \geqslant 2$, there exist constants $c_1$, $c_2$ depending only on $k$ and~$A_0$, such that for any $A\in W^{k,2}\left(\mathcal{K}_B(L)\right),$
there exists a gauge transformation $u \in \mathcal{G}^{k+1,2}_B$ such that $u^\ast A=A_0+\e\alpha,$ where $\alpha\in W^{k,2}\left(\Omega^1_B(M)\right)$ satisfies both the 
gauge condition
$\delta_B\alpha=0$
and the estimate
$$\|\alpha\|_{k,2}\leqslant c_1\|F_A^+\|_{k-1,2}+c_2.$$ 
\end{lem}

\begin{proof}
Let $A=A_0+\sqrt{-1}\alpha_0.$ The basic function $\delta_B\alpha_0$ is $L^2$-orthogonal to the constant functions.
We denote by $\mathbb{G}_B$ the basic Green operator (associated to the basic Laplacian $\Delta_B$) and define $u_0=\exp(\frac{\e}{2}\mathbb{G}_B\left(\delta_B\alpha_0\right)).$
Then $u_0^\ast A=A_0+\e\alpha_0-\e d\mathbb{G}_B\left(\delta_B\alpha_0\right)$. Obviously,
$$\delta_B\left(\alpha_0-d\mathbb{G}_B\left(\delta_B\alpha_0\right)\right)=\delta_B\alpha_0-\delta_Bd\mathbb{G}_B\left(\delta_B\alpha_0\right)=0.$$
To get the estimate, recall that the operator $\delta_B \oplus d_B^+ : \Omega_B^1(M) \to \Omega_B^0(M) \oplus \Omega_B^+(M)$ is a basic transversally elliptic operator. Hence we can apply the basic Hodge decomposition
$$\alpha = \alpha_H + \delta_B \beta + d_B f$$
where $\alpha_H \in \ker (\delta_B \oplus d_B^+)$, $\beta \in \Omega_B^+(M)$ and $f \in \Omega_B^0(M)$. The basic gauge fixing condition $\delta_B \alpha =0$ implies that the last term $d_B f$ vanishes and thus $\alpha = \alpha_H + \delta_B \beta$.
Moreover,
$$F_A^+ = F_{A_0}^+ + \sqrt{-1} d_B^+ \alpha = F_{A_0}^+ + \sqrt{-1} d_B^+ \delta_B \beta = F_{A_0}^+ +\sqrt{-1} \left(\delta_B + d_B^+ \right)\delta_B \beta.$$
By Theorem~\ref{regularity}, we have
\begin{eqnarray*}
\|\delta_B \beta \|_{k,2} &\leqslant&  c_1 \| \left(\delta_B + d_B^+ \right) \delta_B \beta \|_{k-1,2} = c_1 \| F_A^+ - F_{A_0}^+ \|_{k-1,2}\\ 
&\leqslant& c_1 \| F_A^+ \|_{k-1,2} + c_1 \| F_{A_0}^+ \|_{k-1,2}.
\end{eqnarray*}
Now, in order to prove that $\|\alpha_H\|_{k,2}$ is bounded, we use the hypothesis that the foliation $\scrF$ is taut and that $H^1(\scrF)\cap H^{1}(M,\mathbb{Z})$
is a lattice in $H^1(\scrF)$. By Proposition~\ref{complex} the tautness of $\scrF$ implies that $H^1(\scrF)= \ker (\delta_B \oplus d_B^+)$.
Moreover, the condition that $H^1(\scrF)\cap H^{1}(M,\mathbb{Z})$
is a lattice in $H^1(\scrF)$ allows us, up to a an additional basic gauge transformation, to conclude that $\|\alpha_H\|_{k,2}$ can be bounded by a constant
by imitating the standard argument (see~\cite[Claim 5.3.2]{Mor}). 

\end{proof}

\subsection{The bootstrapping procedure}

We would like to use a bootstrapping procedure to turn the $W^{2,2}$-bound on~$\alpha$ and the $C^0$-bound on $\psi$ into $C^\infty$-bounds. In an initial phase we will try to get $W^{3,2}$-bounds on both $\alpha$ and $\psi$ by using Sobolev multiplication on the borderline. As soon as this is achieved, we can repeatedly increment the regularity of both $\alpha$ and $\psi$ by alternating Sobolev multiplication above the borderline with elliptic regularity.\\

Recall that we already established a $W^{2,2}$-bound, and hence in particular a $W^{1,2}$-bound, on~$\alpha$. The basic Sobolev embedding theorem (Theorem~\ref{basic_sobolev_embedding}) tells us that $\alpha$ is bounded in $L^4.$ Since we have a $C^0$-bound on $\psi,$
this implies that $\alpha \bullet \psi \in L^4.$ 
We apply Theorem~\ref{regularity} to get
$$\|\psi\|_{1,4}\leqslant c_1\|\mathcal{D}^{A_0}_B \psi\|_{L^{4}}+c_2 \|\psi\|_{L^{4}}.$$ 
Hence we obtain a $W^{1,4}$-bound on $\psi$ via the Dirac equation.
The basic Sobolev multiplication theorem (Theorem~\ref{basic_sobolev_multiplication}) on the borderline
$$W^{1,2} \times \left(W^{1,4} \cap L^\infty\right) \to W^{1,2}$$
applied to the couple $(\alpha,\psi)$ gives a $W^{1,2}$-bound on $\alpha\bullet\psi$ and hence on $\mathcal{D}^{A_0}_B \psi$. We deduce a $W^{2,2}$-bound on $\psi$ by Theorem~\ref{regularity}. Basic Sobolev multiplication on the borderline (Theorem~\ref{basic_sobolev_multiplication})
$$\left(W^{2,2} \cap L^\infty\right) \times \left(W^{2,2} \cap L^\infty\right) \to \left(W^{2,2} \cap L^\infty\right)$$
applied to the couple $(\psi,\psi)$ gives a $ W^{2,2}$-bound on $\sigma(\psi)$ and therefore by the curvature equation also on $F_A^+$. By Lemma~\ref{gauge_fixing} we get a $W^{3,2}$-bound on $\alpha$. We apply basic Sobolev multiplication (Theorem~\ref{basic_sobolev_multiplication}) above the borderline
$$W^{3,2} \times W^{2,2} \to W^{2,2}$$ 
to the couple $(\alpha,\psi)$ to obtain a $W^{2,2}$-bound on $\alpha \bullet \psi$ and hence on $D_B^{A_0} \psi$. Theorem~\ref{regularity} finally gives a $W^{3,2}$-bound on $\psi$. Having established $W^{3,2}$-bounds on both $\alpha$ and $\psi$, we thus arrived in the stable bootstrapping region. Indeed, since the basic Sobolev multiplication
$$W^{k,2} \times W^{k,2} \to W^{k,2}$$ holds for any $k \geqslant 3,$
we can increment the regularity of $\psi$ and $\alpha$ repeatedly by alternating basic elliptic regularity (Theorem~\ref{regularity}) with the basic Sobolev multiplication
for the curvature equation and Lemma~\ref{gauge_fixing}.

\subsection{Compactness and Hausdorff}

We finally have all the ingredients to prove the main result of this section, compactness of the moduli space $\mathcal{M}_{\rho,\mu}$.

\begin{thm}\label{t:compactness}
Let $M$ be a closed oriented smooth manifold endowed with a Riemannian taut foliation of codimension $4$. 
Suppose moreover that $H^1(\scrF)\cap H^{1}(M,\mathbb{Z})$
is a lattice in $H^1(\scrF)$.
Then, for any \spinc-structure $\rho$ and any perturbation $\mu$, the moduli space $\mathcal{M}_{\rho,\mu}$ is compact in the $C^\infty$-topology.
\end{thm}

\begin{proof}
To prove compactness we need to show that any sequence of Seiberg--Witten solutions $(A_i,\psi_i)$ 
admits a converging subsequence. By the previous results, any solution $(A_i,\psi_i)$ lies automatically in $W^{k,2} \times W^{k,2}$ for all $k$ if written in the appropriate gauge. 
By the basic Rellich theorem, any $W^{k+1,2}$-bounded sequence has a subsequence which converges in $W^{k,2}$. Using this fact and the diagonal process, one can show that there exists a subsequence converging in $W^{k,2}$ for all~$k$. Hence, for any $k$ there is a subsequence converging in $W^{k,2}$.
The basic Sobolev embedding theorem tells us that the subsequence converges in $C^l$ for all~$l$, i.e.~in $C^\infty$.
\end{proof}

The following Lemma, taken from~\cite{Li}, shows that the moduli space $\mathcal{M}_{\rho,\mu}$ is also separated.
\begin{lem}
Suppose that $(A_i,\psi_i)$ and $(A'_i,\psi'_i)$ are two sequences converging to $(A,\psi)$
and $(A',\psi')$ respectively such that, for every $i$, there exists a 
gauge change~$u_i$ 
satisfying $(u_{i}^\ast A_{i},u_{i}{} \psi_{i})=(A'_i,\psi'_i)$.
Then $u_i$ has a subsequence converging to an element $u$ such that $(u^\ast{A},u{ } \psi)=(A',\psi').$
\end{lem}
\begin{proof}
By compactness of $S^1$, $u_i$ is bounded pointwise and in $L^2.$ Moreover, 
\begin{equation*}
2u_i^{-1}du_i=u_i^*(A'_i-A_i)=u_i^*(A'_i-A')+u_i^*(A-A_i)+u_i^*(A'-A),
\end{equation*}
As $A_i' \to A'$ and $A_i \to A$ in $C^\infty$ we deduce that $du_i$ is bounded in $L^2$. 
Now, Theorem~\ref{regularity} implies
$$\|u_i\|_{k+1,2} \leqslant c \|(d_B+\delta_B) u_i\|_{k,2} + c\|u_i\|_{k,2}.$$
Repeating this process, we get a bound on $\|u_i\|_{k,2}$
for all $k$. 
We can then use the basic Rellich theorem, the basic Sobolev embedding theorem and the diagonal procedure as in the proof of Theorem~\ref{t:compactness}
to obtain a subsequence converging in the $C^\infty$-topology to an element $u$ such that $(u^\ast{A},u{ } \psi)=(A',\psi').$
\end{proof}

\section{Transversality and basic Seiberg--Witten invariants}\label{Sec_8}

The linearization of the {{basic gauge group}} action $\mathcal{G}_B$ on a point $(A,\psi)$ (at the point $1$) is 
\begin{eqnarray*}
d\mathfrak{g}_{\rho} :\sqrt{-1}\Omega^0_B(M)&\longrightarrow&\sqrt{-1}\Omega^1_B(M)\times\Gamma_B(\mathcal{W}^+),\\
 \left(\sqrt{-1}f\right) &\longmapsto&(2\sqrt{-1}d_Bf,-\sqrt{-1}f\psi).
\end{eqnarray*}

On the other hand, the linearization of $\mathfrak{sw}_{\rho,\mu}$~\label{SW_map} at a point $(A,\psi)$ is given by the map
\begin{eqnarray*}
d\mathfrak{sw}_{\rho,\mu}:\sqrt{-1}\Omega^1_B(M)\times\Gamma_B(\mathcal{W}^+)&\longrightarrow&\sqrt{-1}\Omega^+_B(M)\times\Gamma_B(\mathcal{W}^-)
\end{eqnarray*}
where
\begin{equation*}d\mathfrak{sw}_{\rho,\mu}|_{(A,\psi)}\left(\sqrt{-1}\alpha,\phi\right)=\left(
\begin{array}{c}
\sqrt{-1}d^+_B\alpha{-\psi\sqrt{-1}\,\bar\phi-\phi\sqrt{-1}\,\bar\psi}-\mathrm{ Re}(\phi,\psi),\\
\mathcal{D}^A_B\phi+\frac{1}{2}\sqrt{-1}\alpha\bullet\psi
\end{array}\right).
\end{equation*}

\begin{lem}
The sequence
\begin{equation*}
0\rightarrow\sqrt{-1}\Omega^0_B(M)\xlongrightarrow{d\mathfrak{g}_{\rho}}\sqrt{-1}\Omega^1_B(M)\times\Gamma_B(\mathcal{W}^+)\xlongrightarrow{d\mathfrak{sw}_{\rho,\mu}}\sqrt{-1}\Omega^+_B(M)\times\Gamma_B(\mathcal{W}^-)\rightarrow 0
\end{equation*}
is a complex over $\mathfrak{m}_{\rho,\mu}$ called {\it{basic Seiberg--Witten complex}}. 
\end{lem}
\begin{proof}
We compute 
\begin{eqnarray*}
&&d\mathfrak{sw}_{\rho,\mu}|_{(A,\psi)}\circ d\mathfrak{g}_\rho(\sqrt{-1}f)\\
&=&(2\sqrt{-1}d^+_B d_Bf {-f\psi\e\,\bar{\e\psi}-\e f\psi\e\,\bar\psi}-\mathrm{ Re}(\sqrt{-1}f\psi,\psi),\\
&& \mathcal{D}^A_B( -\e f\psi)+ \e d_Bf\bullet\psi)\\
&=&(0,-\e \mathcal{D}^A_B(f\psi)+\e d_Bf\bullet\psi).
\end{eqnarray*}
In a local basic $g_{Q}$-orthonormal frame $\{e_1,\cdots,e_m\}$ of $Q$ we have 
\begin{eqnarray*}
\mathcal{D}_A^B(f\psi)&=&\sum_{i} e_i\bullet\nabla^A_{e_i}(f\psi)-\frac{1}{2}\tau_B\bullet(f\psi)\\
&=&f\mathcal{D}^A_B(\psi)+\sum_{i} e_i \bullet (\nabla^T_{e_i}f)\psi\\
&=&f\mathcal{D}^A_B(\psi)+d_Bf \bullet \psi.
\end{eqnarray*}
Note that in the third line we use the fact that $f$ is basic. We deduce that the composition vanishes $$d\mathfrak{sw}_{\rho,\mu}|_{(A,\psi)}\circ d\mathfrak{g}_{\rho}=0$$ and the lemma follows. 
\end{proof}

\begin{lem}
The basic Seiberg--Witten complex is transversally elliptic.
\end{lem}
\begin{proof}
The symbol of a differential operator only depends on the highest order derivatives. The simplified complex splits as the direct sum of two complexes
\begin{equation*}
0\longrightarrow\sqrt{-1}\Omega^0_B(M)\xlongrightarrow{d_B}\sqrt{-1}\Omega^1_B(M)\xlongrightarrow{d^+_B}\sqrt{-1}\Omega^+_B(M)\longrightarrow 0.
\end{equation*}
\begin{equation*}
0\longrightarrow 0\longrightarrow\Gamma_B(\mathcal{W}^+)\xlongrightarrow{\mathcal{D}^A_B}\Gamma_B(\mathcal{W}^-)\longrightarrow 0.
\end{equation*}
Both complexes are transversally elliptic. 
\end{proof}

\begin{rem}
The standard construction of the {{determinant line bundle}}
of a family of elliptic operators can be extended to a family of Fredholm operators (see for instance~\cite{Law-Mic}). The operators $\mathcal{D}^A_B$ and $\delta_B\oplus d_B^+$ are Fredholm. 
The determinant line bundle of the basic Seiberg--Witten complex is then isomorphic to the tensor product of the determinant line bundle of $\mathcal{D}^A_B$ and $\delta_B\oplus d_B^+.$
The complex multiplication generates a nowhere vanishing section trivializing the  determinant line bundle of $\mathcal{D}^A_B$ while a choice of orientation
of $\ker(\delta_B\oplus d_B^+)$ and $\mathrm{coker}(\delta_B\oplus d_B^+)$ (see Proposition~\ref{complex}) trivializes the determinant line bundle of $\delta_B\oplus d_B^+.$ 
Hence, the top exterior power of the tangent to the moduli space is trivialized inducing an orientation on the moduli space.
\end{rem}

Denote by $\mathcal{D}_B^{A,+}$ the restriction of $\mathcal{D}_B^{A}$ to $\Gamma_B(\mathcal{W}^+).$ The (basic) index $\mathrm{ind}_b\left(D_B^{A,+}\right)$ of $D_B^{A,+}$ is a finite integer and its expression is given in~\cite{Bru-Kam-Ric-1,Bru-Kam-Ric}.
Let $\chi_{\mathfrak{sw}}$ be the (real) Euler characteristic of the basic Seiberg--Witten complex.
By Proposition~\ref{complex}, when $\kappa_B=0,$ $\chi_{\mathfrak{sw}}$ is given by
\begin{equation}\label{index}
\chi_{\mathfrak{sw}}=1-\dim H^1(\scrF)+\dim \mathcal{H}^+(\scrF)- 2\,\mathrm{ind}_b\left(D_B^{A,+}\right).
\end{equation}

We define the parametrized basic Seiberg--Witten map:
\begin{eqnarray*}
\mathfrak{psw}_{\rho}:W^{2,2}\left(\mathcal{K}_B(L)\times \Gamma_B(\mathcal{W}^+)\right) \times W^{1,2}\left(\Omega_B^+(M)\right) &\longrightarrow& W^{1,2}\left(\sqrt{-1} \Omega^+_B(M)\times \Gamma_B(\mathcal{W}^-)\right)\\
(A,\psi, \mu)&\longmapsto&\left(F_A^+-\sigma(\psi)-\sqrt{-1}\mu,\mathcal{D}^A_B(\psi)\right).
\end{eqnarray*}

\begin{lem}\label{differential onto}
The differential $d\mathfrak{psw}_{\rho}$ of the parametrized basic Seiberg--Witten map is onto for every $\psi\neq 0.$ 
\end{lem}
\begin{proof}
This is done as for the standard case. Because of the perturbation, we only need to check whether the linearization from $\sqrt{-1}\Omega^1_B(M)\times\Gamma(\mathcal{W}^+)$ to the second component is surjective. At the point $(A,\psi)$ this linearization map equals $L(\sqrt{-1}\alpha,\phi):=\mathcal{D}^A_B\phi+\frac{1}{2}\alpha\bullet\psi.$
Suppose that $\eta\neq 0$ is $L^2$-orthogonal to the image of the map $L$. Then it is $L^2$-orthogonal to the image of $\mathcal{D}^A_B.$
The operator  $\mathcal{D}^A_B$ is self-adjoint, so we have $\mathcal{D}^A_B\eta=0.$ By the basic unique continuation theorem (Theorem~\ref{continuation_thm}),
we can pick an open set $U$ where both $\eta$ and $\psi$ are non-zero.
Moreover, by possibly shrinking $U$, we can choose $\alpha$ such that the real part of $(\sqrt{-1}\alpha\bullet\psi,\eta)$ is strictly positive over $U.$
Modifying $\alpha$ by a bump function supported in $U$, we get a contradiction since $\eta$ is not $L^2$-orthogonal to $L(\sqrt{-1}\alpha,0).$
\end{proof}

A solution $(A,\psi)$ 
to the perturbed basic Seiberg--Witten~(\ref{perturbed_SW}) is called {\it{reducible}} if $\psi \equiv 0.$
For a reducible solution $(A,0)$, 
we have
\begin{eqnarray*}
F_A^+&=\sqrt{-1}\mu.
\end{eqnarray*}
Consider the following subset of $\sqrt{-1} \Omega_B^+(M)$:
\begin{equation}\label{wall}
\Pi_\rho(g_Q):=\{\e\mu\in\e\Omega^+_B(M)\,|\,F_A^+=\sqrt{-1}\mu\,{\text{ for some }}\, A\in\mathcal{K}_B(L)\}.
\end{equation}
We also define
\begin{equation}\label{generic_dim}
P^+:=\{\beta\in\Omega^+_B(M)\,|\,\delta_B\beta=\left(d_B-\kappa_B\wedge\right)\beta=0\}.
\end{equation}
\begin{prop}
The set $\Pi_\rho(g_Q)$ is an affine subspace of codimension 
$\dim P^+.$ In particular,
if $\kappa_B=0$, then $\Pi_\rho(g_Q)$ has codimension $\dim \mathcal{H}^+(\scrF).$
\end{prop}
\begin{proof}
An element $\e\mu\in\e\Omega^+(M)$ is the $\bar\ast$-self-dual part of the curvature $F_A$ if and only if~$\mu$ lies in the image of the map 
\[
\alpha\longmapsto F_{A_0}^++d_B^+\alpha,
\]
where $\alpha\in\Omega^1_B(M).$
The result follows from Proposition~\ref{complex}.
\end{proof}

Combining all the previous results, we obtain

\begin{thm}
Let $M$ be a closed oriented manifold equipped with a taut {{Riemannian foliation}} $\scrF$
of codimension $4$ and a bundle-like metric $g$. 
Suppose that $H^1(\scrF)\cap H^{1}(M,\mathbb{Z})$
is a lattice in $H^1(\scrF)$.
Suppose moreover that the normal bundle~$Q$ admits a \spinc-structure $\rho$.
Then the moduli spaces ${\mathcal{M}}_{\rho,\mu}$ and $\overline{\mathcal{M}}_{\rho,\mu}$ defined in~(\ref{moduli_space}) have the following properties: 
\begin{itemize}
\item  For a generic $\mu$, $\overline{\mathcal{M}}_{\rho,\mu}$ is compact, smooth and of dimension $-\chi_{\mathfrak{sw}}+1$ with a smooth circle action, where $\chi_{\mathfrak{sw}}$ is given by~(\ref{index}).
\item If $\dim \mathcal{H}^+(\scrF)>0,$ then for a generic $\mu,$ ${\mathcal{M}}_{\rho,\mu}$ is smooth and orientable of dimension $-\chi_{\mathfrak{sw}}$.
Moreover, $\overline{\mathcal{M}}_{\rho,\mu}\longrightarrow{\mathcal{M}}_{\rho,\mu}$ is a principal $S^1$-bundle.
\end{itemize}
\end{thm}


We are now in position to define {\it{basic Seiberg--Witten invariants:}}

\begin{defn}
Let $M$ be a closed oriented manifold equipped with a taut {{Riemannian foliation}} $\scrF$
of codimension $4$ and a bundle-like metric $g$. 
Suppose that $H^1(\scrF)\cap H^{1}(M,\mathbb{Z})$
is a lattice in $H^1(\scrF)$.
and that the normal bundle $Q$ admits a \spinc-structure $\rho$.
Suppose moreover that $\dim  \mathcal{H}^+(\scrF)\geqslant 2$ and that $\mu$ is generic such that ${\mathcal{M}}_{\rho,\mu}$ is a closed 
smooth manifold of dimension $-\chi_{\mathfrak{sw}}$ with a chosen orientation.
If $-\chi_{\mathfrak{sw}}$ is odd, then the basic Seiberg--Witten invariant $\mathrm{SW}(\scrF,\rho)=0$. If $-\chi_{\mathfrak{sw}}$ is even, then
\[
\mathrm{SW}(\scrF,\rho)=\int_{{\mathcal{M}}_{\rho,\mu}}e^{\dfrac{-\chi_{\mathfrak{sw}}}{2}},
\]
where $e$ is the Euler class of the principle $S^1$-bundle $\overline{\mathcal{M}}_{\rho,\mu}\longrightarrow{\mathcal{M}}_{\rho,\mu}$.
\end{defn}
\begin{rem}
The value of $\mathrm{SW}(\scrF,\rho)$ is independent of the choice of {$g$} and the generic $\mu$ but depends on the foliation $\scrF$ (see example \ref{depends foliation}) and the transverse \spinc-structure $\rho$.
\end{rem}
\begin{rem}
When $\dim P^+= 1$, the complement of the set $\Pi_\rho(g_Q)$ defined by~(\ref{wall}) has two path connected components called {\it{chambers}}. We distinguish them
in the following way: choose an orientation of $P^+$ and let $\omega\in P^+$ be the unique (positively oriented) form satisfying $\int_M\omega\wedge\omega\wedge\chi_\scrF=1,$
where $\chi_\scrF$ is the {{characteristic form}} of the foliation $\scrF.$ Let $L$ be the determinant line bundle 
of the complex spinor bundle $\mathcal W$.
We define the {\it{basic discriminant}} by
\[
d_\rho(g,\mu)=\int_M (2\pi c_{1}^B(L)+\mu)\wedge\omega\wedge\chi_\scrF,
\]
where $c_{1}^B(L)$ is the basic first Chern class of $L$ (representing the class of the curvature of a basic Hermitian connection on $L$). Remark
that $d_\rho(g_Q,\mu)=0$ if and only if $\mu\in\Pi_\rho(g_Q).$ The {\it{positive}} (resp.\! {\it{negative}}) chamber corresponds to the positive
(resp.\! negative) discriminant $d_\rho(g,\mu).$
\end{rem}

\section{Vanishing theorems}\label{section vanishing}\label{Sec_9}

Throughout this section, let $M$ be a closed oriented manifold equipped with a {{Riemannian foliation}} $\scrF$
of codimension $4$ and a bundle-like metric~$g$. Moreover, suppose that the normal bundle $Q$ admits a \spinc-structure~$\rho$.

\begin{lem}\label{vanishing}
If $s^T$ is positive, then all solutions of the basic Seiberg--Witten equations~(\ref{SW_equations}) have trivial basic spinor.
\end{lem}

\begin{proof}
Let $(A,\psi)$ be a solution of the basic Seiberg--Witten equations~(\ref{SW_equations}). 
If we take the pointwise inner product with $\psi$ in the Weitzenb\"{o}ck formula~(\ref{weitzenbock-2}) and integrate over $M$, then we get
\begin{equation*}
0= \int_M |{\bar{\nabla}^A_{tr}}\ \psi |^2 v_g+ \frac{1}{4}\int_M s^T|\psi|^2v_g + \frac{1}{4} \int_M |\psi|^4v_g,
\end{equation*}
where $v_g$ is the volume form induced by $g.$ 
The lemma follows as $s^T$ is assumed to be positive. 
\end{proof}

Recall that the basic Dirac operator and hence the basic Seiberg--Witten equations depend on the choice of the bundle-like metric $g$.
However, the vanishing condition above depends only on the transverse scalar curvature $s^T$ and so only on the metric $g_Q$. 

\begin{rem}~\label{rem_vanishing}
If we use the Weitzenb\"{o}ck formula~(\ref{weitzenbock}) then the condition we get is $s^T + |\kappa_B|^2 -2\delta_B\kappa_B>0.$
Recall from Remark~\ref{harmonic_remark}, that a suitable conformal change of the metric in the leaf direction can make $\kappa_B$ $\Delta_B$-harmonic.
We obtain then $s^T + |\kappa_B|^2>0$ allowing the transverse scalar curvature to be negative.
\end{rem}

\begin{lem}\label{vanishing_2}
Suppose that $\kappa=\kappa_B$ is basic.   
If $s-\hat{s}+|A|^2+|T|^2$ is positive, then all solutions of the basic Seiberg--Witten equations~(\ref{SW_equations}) have trivial basic spinor. Here $s$ denotes the scalar curvature
of the metric $g$, $\hat{s}$ is the scalar curvature of any leaf whereas 
$A$ and $T$ are the O'Neill tensor fields~\cite{nei}
\begin{eqnarray*}
A_XY&:=&\Pi\left(D^g_{\Pi^\perp(X)}\Pi^\perp(Y)\right)+\Pi^\perp\left(D^g_{\Pi^\perp(X)}\Pi(Y)\right),\\
T_XY&:=&\Pi\left(D^g_{\Pi(X)}\Pi^\perp(Y)\right)+\Pi^\perp\left(D^g_{\Pi(X)}\Pi(Y)\right),
\end{eqnarray*}
for $X,Y$ any vector fields on $M$ and $\Pi:TM\rightarrow T\scrF$, $\Pi^\perp:TM\rightarrow T\scrF^{\perp}.$ 
\end{lem}
\begin{proof}
First, we have that~\cite{Habib-thesis}
\[
|\kappa_B|^2= \delta_B\kappa_B -\delta^T\kappa_B,
\]
where $\delta^T=-\sum_{i=1}^4\iota_{e_i}\nabla^T_{e_i}$ is the transverse codifferential written in a local orthonormal basic frame $\{e_1,\cdots,e_4\}$ of $Q$.
We use Remark~\ref{rem_vanishing} and the relation (see for instance~\cite[Corollary 2.5.19]{bo-ga})
\[
s=s^T+\hat{s}-|A|^2-|T|^2-|\kappa_B|^2-2\delta^T\kappa_B
\]
to deduce the Lemma. 
\end{proof}
\begin{defn}
A Riemannian foliation $\scrF$ has {\it{totally umbilical leaves}} if $$T_XY=\frac{1}{\dim\scrF}~g(X,Y)\tau,$$ for any $X,Y\in\Gamma(T\scrF).$
\end{defn}
\begin{cor}
Suppose that the {{Riemannian foliation}} $\scrF$ of dimension $d=\dim\scrF$ 
has totally umbilical leaves and that $\kappa=\kappa_B$ is basic. Moreover, suppose that
$g$ has constant curvature $c$ such that $12c+\frac{d(d-1)+1}{d^2}|\kappa_B|^2+|A|^2$ is positive. Then all solutions of the basic Seiberg--Witten equations~(\ref{SW_equations}) have trivial basic spinor.
\end{cor}
\begin{proof}
Let $R$ be the curvature of the Levi-Civita connection $D^g$ and $\hat{R}$ be the curvature of the induced connection along the leaves $\hat{D}:=D^g-T$. Then we have the following formula (see~\cite{gra,nei} paying attention to signs in the definition of the curvature tensor)
\[
g(R(X,Y)U,V)=g(\hat{R}(X,Y)U,V)+g(T_XU,T_YV)-g(T_YU,T_XV),
\]
for any $X,Y,U,V\in\Gamma(T\scrF)$.
Using the fact that $\scrF$ has totally umbilical leaves, we deduce that (see also~\cite[Lemma 2.2]{bad-Esc-Ian})
\begin{eqnarray*}
g(R(X,Y)X,Y)&=&g(\hat{R}(X,Y)X,Y)+g(T_XX,T_YY)-g(T_YX,T_XY)\\
&=&g(\hat{R}(X,Y)X,Y)+\frac{1}{d^2}\left(g(X,X)g(Y,Y)-g(X,Y)^2\right) |\kappa_B|^2.
\end{eqnarray*}
In a local orthonormal basis $\{\xi_1,\cdots,\xi_d\}$ of $T\scrF$ we get
\[
\sum_{i,j=1}^d g(R(\xi_i,\xi_j)\xi_i,\xi_j)=\hat{s}+\frac{d^2-d}{d^2}|\kappa_B|^2.
\]
Since the curvature is constant, we deduce that
\[
s-\hat{s}+|A|^2+|T|^2=12c+\frac{d^2-d}{d^2}|\kappa_B|^2+|A|^2+\frac{1}{d^2}|\kappa_B|^2.
\]
This gives us the desired result by Lemma~\ref{vanishing_2}.
\end{proof}
\begin{rem}
Consider a Riemannian manifold with constant curvature which admits a Riemannian foliation with bundle-like metric $g$ and totally umbilical leaves with respect to $g$. A result of~\cite[Proposition 2.4]{bad-Esc-Ian} states that if $\dim\scrF\geqslant 2$, then $\kappa$ is basic and $A\equiv 0.$
\end{rem}


\section{Transverse K\"ahler structures and non-vanishing basic SW-invariants}\label{Sec_10}

Let $M$ be a closed manifold with a {{Riemannian foliation}} $\scrF$ of even codimension~$m$ and $g$ a {\it{bundle-like}} metric on $M$ inducing the metric $g_Q$ on the normal bundle $Q$.
Suppose that $Q$ is equipped with a transverse Hermitian structure $(g_Q,J)$, i.e.~an endomorphism $J:Q\longrightarrow Q$ such that $J^2 = -Id_Q$ and $g_Q(\cdot,\cdot) = g_Q(J\cdot, J\cdot)$. Then the complexified bundle splits as $Q \otimes\mathbb{C}=Q^{1,0}\oplus Q^{0,1},$
where $Q^{1,0}$ (resp.~$Q^{0,1}$) corresponds to the eigenvalue $\sqrt{-1}$ (resp.~$-\sqrt{-1}$).
Similarly $J$ induces a splitting of the complexification of $Q^\ast$:
\[
Q^\ast\otimes\mathbb{C}=\Lambda^{1,0}Q^\ast\oplus\Lambda^{0,1}Q^\ast,
\]
where $\Lambda^{1,0}Q^\ast$ (resp.~$\Lambda^{0,1}Q^\ast$) is the annihilator of $Q^{0,1}$ (resp.~$Q^{1,0}$).
More generally,
the bundle $\Lambda^r_B(M,\mathbb{C})$ of complex basic forms splits as
\begin{equation}\label{split forms}
\Lambda^r_B(M,\mathbb{C})=\bigoplus_{i+j=r}\Lambda^{i,j}_B(M).
\end{equation}
We can define the basic Dolbeault operators $\partial_B:{\Omega}^{p,q}_B(M)\rightarrow\Omega^{p+1,q}_B(M)$ and $\bar\partial_B:{\Omega}^{p,q}_B(M)\rightarrow\Omega^{p,q+1}_B(M)$ by
\[
\partial_B\alpha:=(d_B\alpha)^{p+1,q},\quad\bar\partial_B\alpha:=(d_B\alpha)^{p,q+1}.
\]
The transverse Levi-Civita connection $\nabla^T$ induces canonical Hermitian connections $\nabla^{1,0}$ and $\nabla^{0,1}$ on
$Q^{1,0}$ and $Q^{0,1}$ respectively.  
In a local $g_Q$-orthonormal frame $\{e_1,\cdots,e_m\}$ of $Q$ we can then write
\begin{equation}\label{formulae_del_bar}
\partial_B=\frac{1}{2}\sum_{i=1}^m(e_i+\sqrt{-1}Je_i)^\flat\wedge\nabla^T_{e_i},\quad\bar\partial_B=\frac{1}{2}\sum_{i=1}^m(e_i-\sqrt{-1}Je_i)^\flat\wedge\nabla^T_{e_i}.
\end{equation}
where $\flat$ denotes the $(g_Q,J)$-Hermitian dual. If we denote by $\partial_B^\ast$, $\bar\partial_B^\ast$
the basic adjoints of $\partial_B$, $\bar\partial_B$ with respect to the Hermitian inner product, then
\begin{equation}\label{del_bar_adjoint}
\begin{array}{c}
\partial_B^\ast=-\frac{1}{2}\sum_{i=1}^m\iota_{(e_i-\sqrt{-1}Je_i)}\nabla^T_{e_i}+\frac{1}{2}\iota_{(\tau_B-\sqrt{-1}J\tau_B)},\\
\bar\partial_B^\ast=-\frac{1}{2}\sum_{i=1}^m\iota_{(e_i+\sqrt{-1}Je_i)}\nabla^T_{e_i}+\frac{1}{2}\iota_{(\tau_B+\sqrt{-1}J\tau_B)}.
\end{array}
\end{equation}
Consider the self-adjoint transverse Clifford module
\[
\Lambda^{0,\star}Q^*=\Lambda^{0,\mathrm{even}}Q^\ast\oplus\Lambda^{0,\mathrm{odd}}Q^\ast.
\]
The Clifford action of $X=X^{1,0}+X^{0,1}\in Q^{1,0}\oplus Q^{0,1}$ is given by
\begin{equation}\label{clifford_multiplication}
X\bullet\cdot=\sqrt{2}\left(\left(X^{1,0}\right)^\flat\wedge\cdot-\iota_{X^{0,1}}\cdot\right),
\end{equation}
 The bundle $\wedge^{0,\star}Q^*$ turns out to be a basic Clifford module bundle
and one can then define a basic Dirac operator on it (see for instance~\cite{Kor}).\\

 A {\it{transverse K\"ahler}} structure $(g_Q,J,\omega)$ is given by a transverse Hermitian structure $(g_Q,J)$ such that
$\omega$, the pull-back of $g_Q(J\cdot,\cdot)$ to $M$, is closed.
%
Consider the canonical \spinc-structure $\rho_{can}$ induced by the complex structure $J:Q\rightarrow Q$ on the normal bundle $Q$.
Its determinant bundle is $\Lambda^{0,\frac{m}{2}}Q^\ast$~\cite{Pet}. Let $\rho$ be a \spinc-structure over the normal bundle $Q$ with determinant line bundle $\Lambda^{0,\frac{m}{2}}Q^\ast\otimes E^2$ with $E$ a (foliated) complex Hermitian line bundle over $M.$ The spinor bundle $\mathcal{W}$ can then be identified with $\Lambda^{0,\star}Q^\ast\otimes E.$ Moreover,
\[
\mathcal{W}^+\cong\left(\Lambda^{0,even}Q^\ast\right)\otimes E,\quad\mathcal{W}^-\cong \left(\Lambda^{0,odd}Q^\ast \right) \otimes E.
\]
Let $A_0$ be the canonical connection on $\Lambda^{0,even}Q^\ast$ induced by $\nabla^T$
and fix a basic Hermitian connection $A$ on $E$. The connections $A_0$ and $A$ induce a connection $\mathcal{A}$ on $\mathcal{W}^+.$

The formulae~(\ref{formulae_del_bar}) and~(\ref{del_bar_adjoint}) hold for the basic sections of $\mathcal{W}^+$ using the induced connection $\nabla^\mathcal{A}$
and the Cauchy--Riemann operator $\bar\partial_{{\mathcal{A}}}$ induced by ${\mathcal{A}}.$

\begin{lem}
The basic Dirac operator $\mathcal{D}^{\mathcal{A}}_B:\Gamma_B(\mathcal{W}^+)\longrightarrow \Gamma_B(\mathcal{W}^-)$ becomes
 \[
\mathcal{D}^{\mathcal{A}}_B=\sqrt{2}\left(\bar\partial_{{\mathcal{A}}}+\bar{\partial}_{{\mathcal{A}}}^\ast\right)- \frac{1}{4} \tau_B \bullet.
\]
\end{lem}
\begin{proof}
Let $\{e_1,\cdots,e_m\}$ be a local $g_Q$-orthonormal frame of $Q$. We write using~(\ref{clifford_multiplication})
\begin{eqnarray*}
\mathcal{D}^{\mathcal{A}}_B\psi&=&\sum_{i=1}^m e_i\bullet\nabla^{\mathcal{A}}_{e_i}\psi-\frac{1}{2}\tau_B\bullet\psi\\
&=&\frac{1}{\sqrt{2}}\sum_{i=1}^m\left(\left( e_i-\sqrt{-1}Je_i\right)^\flat\wedge\nabla^{\mathcal{A}}_{e_i}\psi -\iota_{\left(e_i+\sqrt{-1}Je_i\right)}\nabla^{\mathcal{A}}_{e_i}\psi\right)\\
&&-\frac{1}{2\sqrt{2}}\left(\left(\tau_B-\sqrt{-1}J\tau_B\right)^\flat\wedge\psi-\iota_{(\tau_B+\sqrt{-1}J\tau_B)}\psi\right)\\
&=&\sqrt{2}(\bar\partial_\mathcal{A}+\bar\partial_\mathcal{A}^\ast)\psi-\frac{1}{2\sqrt{2}}\left(\left(\tau_B-\sqrt{-1}J\tau_B\right)^\flat\wedge\psi+\iota_{(\tau_B+\sqrt{-1}J\tau_B)}\psi\right)\\
&=&\sqrt{2}\left(\bar\partial_{{\mathcal{A}}}+\bar{\partial}_{{\mathcal{A}}}^\ast\right)\psi- \frac{1}{4} \tau_B \bullet\psi.
\end{eqnarray*}
\end{proof}


We now take a closer look at the $\mathrm{SW}(\scrF,\rho)$ invariants.
Since to define the invariants we impose the tautness condition, we only have to study the system
\begin{equation}\label{new_system}
\left\{\begin{array}{rcl}
\left(\bar\partial_{{{\mathcal{A}}}}+\bar{\partial}_{{{\mathcal{A}}}}^\ast\right)\psi&=&0,\\
F_{\mathcal{{\mathcal{A}}}}^+&=&\sigma(\psi)+\sqrt{-1}\mu.
\end{array}\right.
\end{equation}

As both $F_{\mathcal{A}}^+$ and $\sigma(\psi)+\sqrt{-1} \mu$ are self-dual basic 2-forms, we can split them accordingly to (\ref{split forms}). Let $(\alpha,\beta)\in\Gamma_B(\mathcal{W}^+),$ where $\mathcal{W}^+\cong\left(\Lambda^{0,0}Q^\ast\oplus\Lambda^{0,2}Q^\ast\right)\otimes E$. Then the curvature equation becomes (see for instance \cite{Mor,Sal})
\begin{equation}
\left\{\begin{array}{rcl}
F_A^{0,2}&=&\frac{1}{2} \bar{\alpha}{\beta},\\
\left(F_{\mathcal{A}}^+\right)^{1,1}&=&\frac{\sqrt{-1}}{4}\left(|\alpha|^2 - |\beta|^2 \right) \omega + \sqrt{-1}\mu.
\end{array}\right.
\end{equation}
If we choose the perturbation $\mu$ to lie in the $(1,1)$-part, then the system of equations~(\ref{new_system}) is equivalent to
\begin{equation}\label{system}
\left\{\begin{array}{rlcr}
F_A^{0,2}=0, && & \bar\alpha\beta=0,\\
\bar\partial_{\mathcal{A}}\alpha=0,&& &-\bar\partial_{\mathcal{A}}^\ast\beta=0,\\
4\left(F_{A_0}^++F_A^+ -\sqrt{-1}\mu\right)^{1,1}&&=&{\e}(|\alpha|^2-|\beta|^2)\,\omega.
\end{array}\right.
\end{equation}
Indeed, as the foliation is transversally K\"{a}hler, the relation $\bar\partial_{\mathcal{A}}  \bar\partial_{\mathcal{A}} = F_{\mathcal{A}}^{0,2}$ still holds and we can conclude as in the 4-dimensional case.
Moreover, we deduce from the basic unique continuation theorem (Theorem~\ref{continuation_thm}) that either $\alpha \equiv 0$ and/or $\beta \equiv 0.$
We can now use the splitted system of equations to write down a non-vanishing theorem just as it is done for K\"{a}hler surfaces. But first, we need the following
Weitzenb\"{o}ck formula:

\begin{lem}\label{Weitzenbock_2} 
Let $\alpha\in\Gamma_B(\Lambda^{0,0}Q^\ast\otimes E)$. Then
\[
\bar\partial^\ast_{\mathcal{A}}\bar\partial_{\mathcal{A}}\alpha=\frac{1}{2}{d_{{\mathcal{A}}}}{d_{\mathcal{A}}^\ast}\alpha-{\sqrt{-1}}\,\left(F_A,\omega\right)\alpha+\frac{\e}{2}\nabla^{\mathcal{A}}_{J\tau_B}\alpha.
\]
\end{lem}
\begin{proof}
Let $\{e_1,\cdots,e_m\}$ be a local $g_Q$-orthonormal frame of $Q$. From~(\ref{formulae_del_bar}) and~(\ref{del_bar_adjoint}) we get:
\begin{eqnarray*}
\bar\partial^\ast_{\mathcal{A}}\bar\partial_{\mathcal{A}}\alpha&=&\frac{1}{2}\bar\partial^\ast_{\mathcal{A}}\left(\sum_{i=1}^m\left(\nabla^{\mathcal{A}}_{e_i}\alpha\right)(e_i-\e Je_i)^\flat\right)\\
&=&-\frac{1}{4}\sum_{i,j=1}^m\iota_{e_j+\e Je_j}\nabla^{\mathcal{A}}_{e_j}\left(\left(\nabla^{\mathcal{A}}_{e_i}\alpha\right)(e_i-\e Je_i)^\flat\right)\\
&&+\frac{1}{4}\sum_{i=1}^m\iota_{\tau_B+\e J\tau_B}\left(\left(\nabla^{\mathcal{A}}_{e_i}\alpha\right)(e_i-\e Je_i)^\flat\right)\\
&=&-\frac{1}{2}\sum_{i=1}^m\nabla^{\mathcal{A}}_{e_i}\nabla^{\mathcal{A}}_{e_i}\alpha+\frac{\e}{2}\sum_{i=1}^m\nabla^{\mathcal{A}}_{Je_i}\nabla^{\mathcal{A}}_{e_i}\alpha+\frac{1}{2}\sum_{i=1}^m\nabla^{\mathcal{A}}_{\nabla^T_{e_i}e_i}\alpha\\
&&+\frac{\e}{2}\sum_{i=1}^m\nabla^{\mathcal{A}}_{\nabla^T_{e_i}Je_i}\alpha+\frac{1}{2}\nabla^{\mathcal{A}}_{\tau_B}\alpha+\frac{\e}{2}\nabla^{\mathcal{A}}_{J\tau_B}\alpha\\
&=&-\frac{1}{2}\sum_{i=1}^m\nabla^{\mathcal{A}}_{e_i}\nabla^{\mathcal{A}}_{e_i}\alpha+\frac{1}{2}\sum_{i=1}^m\nabla^{\mathcal{A}}_{\nabla^T_{e_i}e_i}\alpha+\frac{1}{2}\nabla^{\mathcal{A}}_{\tau_B}\alpha\\
&&-\frac{\e}{2}\sum_{i=1}^m\nabla^{\mathcal{A}}_{e_i}\nabla^{\mathcal{A}}_{Je_i}\alpha+\sum_{i=1}^m\frac{\e}{2}\nabla^{\mathcal{A}}_{\nabla^T_{e_i}Je_i}\alpha+\frac{\e}{2}\nabla^{\mathcal{A}}_{J\tau_B}\alpha \\
&=&\frac{1}{2}d_{\mathcal{A}}^\ast d_{\mathcal{A}}\alpha-\e \,(F_A,\omega)\alpha+\frac{\e}{2}\nabla^{\mathcal{A}}_{J\tau_B}\alpha.
\end{eqnarray*}
\end{proof}

\begin{thm}\label{nonvanishing}
Let $M$ be a closed manifold equipped with a taut {{Riemannian foliation}} $\scrF$
of codimension $4$, a bundle-like metric $g$ and a transverse K\"ahler structure $(g_Q,J,\omega)$ such that $H^1(\scrF)\cap H^{1}(M,\mathbb{Z})$
is a lattice in $H^1(\scrF)$ and $\dim  \mathcal{H}^+(\scrF)\geqslant 2$. Then, $\mathrm{SW}(\scrF,\rho_{can})=\pm1.$ 
\end{thm}

\begin{proof}
The canonical \spinc-structure $\rho_{can}$ corresponds to the case where $E=\mathbb{C}$ is the trivial line bundle. 
We pick the perturbation $\sqrt{-1}\mu=F_{A_0}^+ - \frac{\e}{4} \omega$.

First, we claim that $\alpha$ cannot vanish. Indeed, the last equation of~(\ref{system}) reads
\begin{equation}\label{curvature_exp}
\left(d_B^+A\right)^{1,1}=-\frac{\e}{4}(1-|\alpha|^2+|\beta|^2)\,\omega,
\end{equation}
where $A\in\e\Omega_B^1(M).$ 
Let $\theta\in\Gamma_B(\Lambda^{0,0}Q^\ast)$ be a non zero  holomorphic section with respect to the holomorphic connection ${\mathcal{A}}$, i.e.
$\bar\partial_{\mathcal{A}}\theta=0.$ The existence of $\theta$ is a consequence of Theorem~\ref{regularity}. Actually, $\theta$ is a nowhere vanishing section since the line bundle is trivial. 
Then, $A^{0,1}=-\bar\partial_B\log |\theta|^2$ and thus $A=\frac{1}{2}\left(\left(\partial_B-\bar\partial_B\right)\log |\theta|^2\right).$
In particular, $d_BA=\bar\partial_B\partial_B\left(\log |\theta|^2\right)=\frac{\sqrt{-1}}{2}d_BJd_B\left(\log |\theta|^2\right).$ The inner product with~$\omega$ gives
\begin{eqnarray*}
\langle\left(d^+_BA\right)^{1,1},\omega\rangle&=&\langle d_B A,\omega\rangle
=\frac{\sqrt{-1}}{2}\langle Jd_B\left(\log |\theta|^2\right),\delta_B\omega\rangle\\
&=&\frac{\sqrt{-1}}{2}\langle Jd_B\left(\log |\theta|^2\right), J{\kappa_B} \rangle
=0. 
\end{eqnarray*}

Using~(\ref{curvature_exp}), this implies that $\int_M\left(1-|\alpha|^2+|\beta|^2\right) v_g=0 $ and thus $\alpha$ cannot vanish. We deduce that $\beta \equiv 0$.
Now, Lemma~\ref{Weitzenbock_2} and~(\ref{curvature_exp}) imply that
\begin{equation*}
\frac{1}{2}\int_M |d_{\mathcal{A}}\alpha|^2 v_g -\frac{1}{4} \int_M (1-|\alpha|^2)\,|\alpha|^2 v_g=0. 
\end{equation*}
As $\int_M\left(1-|\alpha|^2\right) v_g=0,$ we can rewrite this as
\begin{equation}\label{Weit_integrated}
\frac{1}{2}\int_M |d_{\mathcal{A}}\alpha|^2 v_g +\frac{1}{4} \int_M (1-|\alpha|^2)^2 v_g=0. 
\end{equation}

From~(\ref{Weit_integrated}), we conclude that $d_{\mathcal{A}}\alpha=0$ and $|\alpha|\equiv 1$. This implies that $A=\frac{1}{2}\left(\left(\partial_B-\bar\partial_B\right)\log |\alpha|^2\right)=0$
and so $\mathcal{A}$ is the exterior derivative $d$. Hence, $\alpha$ is a constant and after a gauge transformation, $\alpha\equiv 1.$



It remains to be proven that the moduli space is smooth at this point. The problem is reduced to check that the kernel of $d\mathfrak{sw}_{\rho_{can},\mu}|_{(d,(1,0))}$ is trivial.
We are looking to the solutions $(\e\lambda,(\phi_0,\phi_1))$ of the system
\begin{eqnarray}
d^+_B(\e\lambda)-(\bar{\phi}_1-{\phi}_1)-\e\phi_0\,\omega&=&0,\label{equa_1}\\
 \bar\partial_B \phi_0+\bar\partial^\ast_B \phi_1+(\e\lambda)^{0,1}&=&0\label{equa_2}.
\end{eqnarray}
By gauge transformation, we are assuming here that $\phi_0$ is real. We apply $\bar\partial_B$ to the equation $(\ref{equa_2})$ to obtain
\[
\bar\partial_B\bar\partial^\ast_B \phi_1+\bar\partial_B(\e\lambda)^{0,1}=0,
\]
while from the equation $(\ref{equa_1}),$ we have $\bar\partial_B(\e\lambda)^{0,1}=\phi_1.$ Hence, 
\[
\bar\partial_B\bar\partial^\ast_B \phi_1+\phi_1=0.
\]
Coupling with $\phi_1$ and integrating, we deduce that $\phi_1=0.$ Now, the system becomes
\begin{eqnarray}
\frac{1}{2}(d_B\lambda,\omega)-\phi_0&=&0,\label{equa_3}\\
 \bar\partial_B \phi_0+(\e\lambda)^{0,1}&=&0.\label{equa_4}
\end{eqnarray}
The equation~(\ref{equa_4}) implies that $d_B\lambda=2\e\,\bar\partial_B\partial_B \phi_0.$ Plugging this in equation~(\ref{equa_3}) and by a simple computation
we obtain that $-\Delta_B\phi_0+\nabla^T_{\tau_B}\phi_0-\phi_0=0.$ Coupling with $\phi_0$ and integrating, we obtain
\[
\int_M-|d_B\phi_0|^2+\frac{1}{2}\left(d_B|\phi_0|^2\right)(\tau_B)-|\phi_0|^2=0.
\]
However, as $\tau_B=0,$ we conclude that $\phi_0=0$ and thus $\lambda=0.$
\end{proof}

\begin{rem}\label{taubes_result}
A similar non-vanishing result should hold when the transverse structure $(g_Q,J,\omega)$ is almost-K\"ahler i.e.~when $(g_Q,J)$ is almost-Hermitian and $\omega$ is closed~\cite{tau}.
Moreover, the tautness condition is not needed in the above result.
\end{rem}

\section{$K$-contact and Sasakian 5-manifolds}\label{Sec_11}
%

In this section, we take a closer look at some particular class of manifolds with Riemannian foliations: Sasakian manifolds
and more generally $K$-contact manifolds.
Let $(M,\eta)$ be a {\it{contact manifold}} of dimension $2n+1$, where $\eta$ is the \emph{contact $1$-form} satisfying $\eta\wedge(d\eta)^n\neq 0$ at every point of $M$.
The Reeb vector field $\xi \in \Gamma(TM)$ is uniquely determined by
\begin{equation*}
\eta(\xi)=1,\quad \iota_\xi\left(d\eta\right)=0.
 \end{equation*}
The corresponding \emph{foliation} $\scrF_\xi$ is the distribution $\mathbb{R}\xi \subset TM.$ 

\begin{defn}\label{def:K-contact}
A \emph{$K$-contact structure} $(\eta,\xi,\Phi,g)$ on~$M$ consists of a contact form~$\eta$ with Reeb field $\xi$ together with an endomorphism $\Phi:TM\rightarrow TM$ satisfying
\begin{equation*}
\quad\Phi^2=-\mathrm{Id}_{TM}+\xi\otimes\eta,\quad \mathcal{L}_\xi \Phi = 0.
\end{equation*}
We also require the following compatibility conditions with $\eta$:
\begin{equation*}
d\eta(\Phi X,\Phi Y)=d\eta(X,Y),\quad d\eta(Z,\Phi Z)>0\quad\forall X,Y \in TM, Z\in \ker(\eta) \setminus \{0\}.
\end{equation*}
This defines a Riemannian metric on $M$ given by
\begin{equation*}
g(X,Y)=d\eta(X,\Phi Y)+\eta(X)\eta(Y).
\end{equation*}
\end{defn}

For a $K$-contact structure $(\eta,\xi,\Phi,g)$, the foliation $\scrF_\xi$ is Riemannian and the leaves of $\scrF_\xi$ are geodesics. In particular, the foliation is taut.
Moreover, on closed K-contact manifolds, $H^1(\scrF)$ is isomorphic to
$H^{1}(M,\mathbb{R})$~\cite[Proposition 7.2.3]{bo-ga}. In particular, $H^1(\scrF)\cap H^{1}(M,\mathbb{Z})$
is a lattice in $H^1(\scrF)$. So, our hypothesis to obtain a compact moduli space is satisfied.

\begin{defn}
A {\it{Sasakian structure}} is a $K$-contact structure $(\eta,\xi,\Phi,g)$ satisfying the integrability condition 
\[
D^g_X\Phi=\xi\otimes X^{\flat_g}-X\otimes\eta,\quad \forall X\in \Gamma(TM),
\]
where $D^g$ is the Levi-Civita connection.
\end{defn}

Clearly, a $K$-contact (resp.~Sasakian) structure induces an almost-K\"ahler (resp. \!K\"ahler) structure on the normal bundle $Q.$
In particular, $\Phi$ induces an almost-complex (resp. ~complex) structure $J$ on $Q$ and hence a canonical \spinc-structure.
We can now try to apply the vanishing theorems of Section~\ref{Sec_9} to this particular setup:
\begin{cor}
Let $(M,\eta,\xi,\Phi,g)$ be a closed $K$-contact 5-manifold. If the transverse scalar curvature $s^T$ of $g$ is positive, then all the invariants $SW(\scrF_\xi,\rho)$ vanish.
\end{cor}

\begin{ex}
If the metric $g$ of the $K$-contact structure is Einstein, then the structure is in fact Sasakian and the metric has positive scalar curvature~\cite{Apo-Dra-Mor,boy-gal-1,bo-ga}.
Also, whenever a $K$-contact manifold is locally symmetric or conformally flat, then it is Sasakian and has constant positive curvature (see~\cite{Oku,Tan,Tan-1}).
We conclude that in all of the above cases all the basic Seiberg--Witten invariants vanish.
\end{ex}
%



On the other hand, we can also apply the non-vanishing result of Theorem~\ref{nonvanishing}: 

\begin{cor}
Let $(M,\eta,\xi,\Phi,g)$ be a compact 5-dimensional Sasakian manifold with $\dim\mathcal{H}^+(\scrF)\geqslant 2$. Then $\mathrm{SW}(\scrF_\xi,\rho_{can})=\pm1.$ 
\end{cor}

Finally, we can combine the vanishing and non-vanishing results in order to obtain obstructions to metrics with positive scalar curvature or Einstein metrics.

\begin{cor}\label{obstruction}
Let $(M,\eta,\xi,\Phi,g)$ be a $5$-dimensional compact Sasakian manifold with $\dim\mathcal{H}^+(\scrF_\xi)\geqslant 2$. Then $M$ does not admit any bundle-like metric with respect to $\scrF_\xi$ of positive transverse scalar curvature.
In particular, $(M,\eta)$ does not admit a Sasaki--Einstein metric.
\end{cor}

When the Sasakian structure is regular (see~\cite[Definition 6.1.25]{bo-ga}), it is clear that the basic Seiberg--Witten invariants are the same as the classical Seiberg--Witten invariants of the smooth 4-manifold given by the leaf space. This implies that the invariants really depend on the foliation and not only on the diffeomorphism type of the manifold as shown in the following simple example:

\begin{ex}\label{depends foliation}

The 5-manifold $8(S^2 \times S^3)$ admits two regular Sasakian structures: a negative one (see~\cite[Definition 7.5.24]{bo-ga}) coming from a circle bundle over a Barlow surface (see~\cite{Bar} and~\cite[Example 10.4.6]{bo-ga}) and a positive one coming from a circle bundle over a del Pezzo surface $\mathbb{CP}^2 \# 8\overline{\mathbb{CP}^2}$ (see~\cite[Proposition 10.4.4]{bo-ga}). In the first case, there are two non-vanishing classes corresponding to the canonical and the anti-canonical class (see for instance~\cite{Ozs-Sza}), whilst in the second case all the basic Seiberg--Witten invariants vanish.
\end{ex}

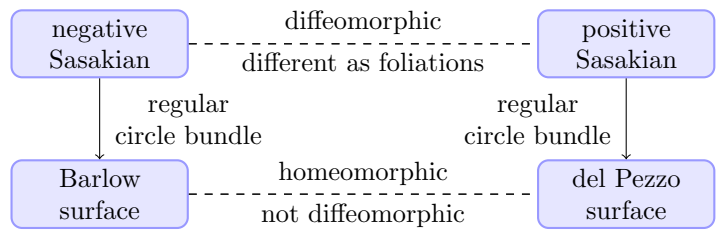
\begin{center}
\begin{tikzpicture}
 \tikzstyle{block}=[rectangle, draw=blue!40, thick, fill=blue!10, text width=6em, text centered, rounded corners, minimum height=2em];
  \path (0,0) node [block] (left up) {negative Sasakian}
            (7,0) node [block] (right up) {positive Sasakian}
	   (0,-2) node [block] (left down) {Barlow surface}
            (7,-2) node [block] (right down) {del Pezzo surface};
  \draw[dashed] (left up) -- node[above,text width=8em, text centered]{diffeomorphic}(right up);
  \draw[dashed] (left up) -- node[below, text centered]{different as foliations}(right up);
  \draw[->] (right up) -- node[left,text width=6em, text centered]{regular circle bundle}(right down);
  \draw[->] (left up) -- node[right,text width=6em, text centered]{regular circle bundle}(left down);
  \draw[dashed] (left down) -- node [below, text centered]{not diffeomorphic}(right down);
  \draw[dashed] (left down) -- node [above,text centered]{homeomorphic}(right down);
\end{tikzpicture}
\end{center}

\begin{rem}
On a closed simply-connected 5-manifold, the space of \spinc-structures is isomorphic to the space of homotopy classes of almost-contact structures. Whilst this isomorphism is not canonical in general, it becomes so on a contact metric manifold. 
We conclude that on a closed simply-connected $K$-contact 5-manifold $(M,\eta,\xi,\Phi,g)$, the basic Seiberg--Witten map may be interpreted as a map eating homotopy classes of almost-contact structures instead of \spinc-structures:
$$SW(\scrF_\xi) :  \{\text{homotopy classes of almost-contact structures}\} \rightarrow \mathbb{Z}.$$
Furthermore, Geiges showed that every homotopy class of almost contact structures contains at least one contact structure~\cite{Gei}.
\end{rem}

\section{Applications}

We believe the construction of foliated Seiberg--Witten invariants lays the groundwork for some future applications which we now briefly describe.

Firstly, the invariant can be used as an obstruction to the existence of transverse Einstein metrics compatible with a given foliation structure. Indeed, Corollary~\ref{obstruction} provides a new obstruction to the existence of Sasaki--Einstein metrics on a given compact Sasakian 5-manifold. In the same vein, Remark~\ref{taubes_result} can be used to say that a $5$-dimensional compact contact manifold $(M,\eta,\xi)$
equipped with a bundle-like Riemannian metric of positive transverse scalar curvature and satisfying $\dim\mathcal{H}^+(\scrF_\xi)\geqslant 2$ should not admit any compatible K-contact structure. Pushing this line of thought further, one may try to apply the Atiyah--Singer index formula for basic transversally elliptic operators~\cite{Bru-Kam-Ric, Bru-Kam-Ric-1} to obtain a LeBrun--Hitchin--Thorpe inequality ~\cite{leb-2} for transverse Einstein metrics on Riemannian foliations.

A second important application that we can point out is that our invariants may allow us to define Seiberg--Witten invariants for {\it{orbifolds}} (for an introduction to orbifold theory, we refer the reader to~\cite{Ade-Lei-Rua}). This approach looks particularly interesting as, to the best of our knowledge, an orbifold version of Seiberg--Witten theory seems to be only known for symplectic $4$-orbifolds (see~\cite{che,che-1}) and it can be used for example to prove the non-existence of Einstein metrics on $4$-orbifolds (see for instance~\cite{leb,sun}). To define foliated Seiberg--Witten invariants for orbifolds, recall that any orbifold can be represented as the space of leaves of a Riemannian foliation $(M,\scrF)$ with compact leaves. The classical construction due to Satake is to pick a Riemannian metric on the orbifold and to consider the orthonormal frame bundle on the orbifold. The action of the orthogonal group on the orthonormal frame bundle is locally free and the quotient space is precisely the orbifold (for more details see~\cite{Kaw,Kaw-1,Moe-Mrc}).
Still, it remains an open question to check whether $H^1(\scrF)\cap H^{1}(M,\mathbb{Z})$ is a lattice in $H^1(\scrF)$. It has been shown~\cite{Sat} that a version of de Rham cohomology for orbifolds is isomorphic to its singular cohomology. Hence it makes sense to speak about integral singular cohomology $H^1(\scrF,\mathbb{Z})$. However, it is not clear to us whether the condition is satsified.

A third possible application of the foliated Seiberg--Witten invariants is to distinguish {\it{homotopic}} but non {\it{isotopic}} Riemannian foliations. Unfortunately, we did not succeed to find any examples that are not known to be different by other methods.

Lastly, we conclude this discussion by pointing out that the original motivation for this work was to study the following conjecture of LeBrun: on an oriented compact $4$-manifold equipped with an anti-self-dual Einstein metric of negative scalar curvature all the Seiberg--Witten invariants vanish~\cite{leb-1}. Proving such a result seems out of reach by current methods but a possible line of attack would be to define Seiberg--Witten invariants on the twistor space and link these to the Seiberg--Witten invariants on the base 4-manifold. Being able to define Seiberg--Witten invariants on Riemannian foliations looks like a promising first step in this direction.

\end{document}